\documentclass[12pt]{amsart}
\pdfoutput=1
\usepackage{amssymb,amscd}
\usepackage{amsmath,amsfonts}
\usepackage{graphicx}
\usepackage{mathrsfs}
\usepackage{dsfont}
\usepackage[left=1.4in,right=1.4in,bottom=1.3in,top=1.3in]{geometry}
\PassOptionsToPackage{hyphens}{url}
\usepackage[all,2cell]{xy}
\UseTwocells

\newcommand{\ObSchS}{{\rm{Ob}}((Sch/S)_{fppf})}

\newcommand{\calc}{\mathcal C}
\newcommand{\Ob}{{\rm{Ob}}}

\newcommand{\Z}{\mathbb Z}

\newcommand{\N}{\mathbb N}

 \theoremstyle{plain}
\newtheorem{theorem}{Theorem}[section]
\newtheorem{corollary}[theorem]{Corollary}
\newtheorem{lemma}[theorem]{Lemma}
\newtheorem{proposition}[theorem]{Proposition}
\newtheorem{definition-proposition}[theorem]{Definition/Proposition}

\newtheorem{definition}[theorem]{Definition}

 \theoremstyle{definition}

\newtheorem{definition1}[theorem]{Definition}

\theoremstyle{remark}

\newtheorem{remark}[theorem]{Remark}

\numberwithin{equation}{section}

\newcommand{\lemref}[1]{Lemma~\ref{#1}}

\newcommand{\proref}[1]{Proposition~\ref{#1}}

\newcommand{\defref}[1]{Definition~\ref{#1}}

\makeatletter
\def\@seccntformat#1{\@ifundefined{#1@cntformat}%
    {\csname the#1\endcsname\quad}
    {\csname #1@cntformat\endcsname}}
\newcommand{\section@cntformat}{\S\thesection.\enspace}
\newcommand{\subsection@cntformat}{\S\thesubsection.\enspace}
\newcommand{\subsubsection@cntformat}{\S\thesubsubsection\enspace}
\makeatletter

\usepackage[dvipsnames, svgnames, x11names]{xcolor}
\definecolor{cite}{rgb}{0.50,0.00,1.00}
\definecolor{url}{rgb}{0.00,0.50,0.75}
\definecolor{link}{rgb}{0.00,0.00,0.50}
\usepackage[colorlinks,linkcolor=RoyalBlue,urlcolor=url,citecolor=magenta,breaklinks]{hyperref}

\usepackage{tikz-cd}
\usepackage{capt-of}

\makeatletter
\@namedef{subjclassname@2020}{\textup{2020} Mathematics Subject Classification}
\makeatother

\begin{document}
\title{Perfect algebraic spaces and perfect morphisms}

\author{Tianwei Liang}
\email{(Tianwei Liang) 2015200055@e.gzhu.edu.cn}

\begin{abstract}
We develop a theory of perfect algebraic spaces that extend the so-called perfect schemes to the setting of algebraic spaces. We prove several desired properties of perfect algebraic spaces. This extends some previous results of perfect schemes, including the recent one developed by Bertapelle et al. in \cite{Bertapellea}. Moreover, our theory extends the previous one developed by Xinwen Zhu in \cite{Zhu}. The perfect morphisms will provide equivalent descriptions to perfect algebraic spaces.

Our method to define perfect algebraic spaces differs from all previous approaches, as we utilize representability of functors. There is a natural notion of algebraic Frobenius morphisms of algebraic spaces, which is analogous to the absolute Frobenius morphisms of schemes. In terms of algebraic Frobenius morphisms, one can define perfection of an algebraic space.
\end{abstract}

\subjclass[2020]{Primary 14A20, 18A23, 14A15}

\keywords{Perfect algebraic spaces; perfect schemes; perfect rings; Frobenius morphism; category theory.}
\thanks{The author would like to thank F.N. Meng for giving some advice.}

\date{\today}
\maketitle

\tableofcontents
\section{Introduction}
Let $p$ be a prime number and $\mathbb{F}_{p}$ a finite field of order $p$.
\subsection{Background and motivation}
The notion of perfect rings is particularly important in both commutative algebra and algebraic geometry. It acts as the foundations of many related fileds. For example, in classical commutative algebra, one has further notions of perfect algebras and perfect fields, which lead to many classic results in algebraic geometry, especially the recent works in \cite{Scholze1,Scholze2}. In some sense, the definitions of perfectness almost completely depends on the existence of so-called Frobenius morphisms. Here we take two familiar cases as an example.

Assume that $A$ is a ring of characteristic $p$. We say that $A$ is \textit{perfect} if the Frobenius endomorphism $\phi:A\rightarrow A,a\mapsto a^{p}$ is an isomorphism. This can be globalized to the setting of schemes. Let $X$ be an $\mathbb{F}_{p}$-scheme. We say that $X$ is \textit{perfect} if the absolute Frobenius morphism $\Phi:X\rightarrow X$ is an isomorphism. The notion of perfect schemes becomes more and more important in algebraic geometry in recent years. Because of its wonderful properties, perfect schemes have become foundations of some important areas of algebraic geometry (see \cite{Scholze, Zhu, Zhu1}).

The prototype of perfect schemes is the so-called perfect varieties that first arose in Serre's classical paper \cite{Serre}. However, a perfect variety is, in general, not a scheme. In \cite{Greenberg}, Greenberg first constructed the perfect closure of an $\mathbb{F}_{p}$-scheme, and in particular, the perfect closure of a ring. At that time, the notion of perfect scheme made its official debut. The perfection of rings or schemes gives rise to the so-called perfection functor. The perfection functor has played a significant role in many areas in algebraic geometry (see, for example, \cite{Bertapelle1,Liu,Boyarchenko}).

The theory of perfect algebraic spaces has been developed recently by Liang Xiao and Xinwen Zhu in \cite{Zhu,Zhu1}. Perfect algebraic spaces generalize perfect schemes to a new step. However, their theory is restricted to algebraic spaces over a perfect field $k$ of characteristic $p$. Meanwhile, they do not mention some fundamental properties of perfect algebraic spaces like taking fibre products, as they are irrelevant in \cite{Zhu,Zhu1}. Another limitation of their theory is that they follow the original method to define perfect algebraic spaces. Their definition is limited in the framework of Frobenius morphisms. More precisely, they avoid the direct Frobenius morphism $F\rightarrow F$, but turn to Frobenius morphisms of $k$-algebras whose Frobenius morphisms make sense. Then the Frobenius morphism $\sigma_{F}:F\rightarrow F$ of $F$ is given by the evaluation on the Frobenius morphisms of $k$-algebras. And $F$ is \textit{perfect} if the Frobenius morphism $\sigma_{F}:F\rightarrow F$ is an isomorphism.

\subsection{Main results}
In this paper, we develop a theory that provides desired extensions of previous works. The direct notion of Frobenius morphism $F\rightarrow F$ of an algebraic space $F$ is vague, since $F$ is a sheaf of sets, and it is impossible to define Frobenius morphism of a set. In order to replace utilizing any kinds of Frobenius morphism $F\rightarrow F$ to define perfect algebraic spaces, we focus on the functoriality of algebraic spaces making use of representability of functors. This makes our approach differs from all previous ones.

We first start by upgrading the well-known theories of perfect rings and perfect schemes. This enables us to apply some nice properties of perfect schemes to the study of perfect algebraic spaces. Our approach generates four types of perfect algebraic spaces, named \textit{perfect}, \textit{quasi-perfect}, \textit{semiperfect}, and \textit{strongly perfect algebraic spaces}.

Let $S$ be some base scheme and $AP_{S}$ be the category of algebraic spaces over $S$. Let $\textrm{Perf}_{AP_{S}}$, $\textrm{QPerf}_{AP_{S}}$, $\textrm{SPerf}_{AP_{S}}$, and $\textrm{StPerf}_{AP_{S}}$ denote the categories of these four types of perfect algebraic spaces respectively. We state the relations between these four types of perfect algebraic spaces and have the following full embedding
\begin{align}
\textrm{Perf}_{AP_{S}}\subset AP_{S},
\end{align}
together with a string of full embeddings
\begin{align}
\textrm{StPerf}_{AP_{S}}\subset \textrm{QPerf}_{AP_{S}}\subset \textrm{SPerf}_{AP_{S}}\subset AP_{S}.
\end{align}
It seems that $\textrm{Perf}_{AP_{S}}$ is disjoint with other categories. However, there is some relation between their subcategories.

Let $\widehat{{\rm{Perf}}_{AP_{S}}}$, $\widehat{{\rm{QPerf}}_{AP_{S}}}$, $\widehat{{\rm{SPerf}}_{AP_{S}}}$, and $\widehat{{\rm{StPerf}}_{AP_{S}}}$ be the subcategories of $\textrm{Perf}_{AP_{S}}$, $\textrm{QPerf}_{AP_{S}}$, $\textrm{SPerf}_{AP_{S}}$, and $\textrm{StPerf}_{AP_{S}}$ of representable perfect algebraic spaces. Then there is a string of full embeddings
\begin{align}
\widehat{{\rm{Perf}}_{AP_{S}}}\subset\widehat{{\rm{StPerf}}_{AP_{S}}}\subset\widehat{{\rm{QPerf}}_{AP_{S}}}\subset\widehat{{\rm{SPerf}}_{AP_{S}}}. \label{T2}
\end{align}

Perfect algebraic spaces satisfy the following statement on fibre products.
\begin{proposition}
Let $F\rightarrow H$ and $G\rightarrow H$ be morphisms of algebraic spaces over $S$. If $F,G$ are perfect and $H$ is quasi-perfect satisfying some extra condition, then the fibre product $F\times_{H}G$ is a perfect algebraic space.
\end{proposition}

Moreover, the categories of other types of perfect algebraic spaces enjoys the following desired property on fibre products.
\begin{proposition}
The categories $\rm{QPerf}_{AP_{S}}$, $\rm{SPerf}_{AP_{S}}$, and $\rm{StPerf}_{AP_{S}}$ are all stable under fibre products.
\end{proposition}
These generalize the previous result in \cite[\S7]{Greenberg} that the category $\textrm{Perf}/\mathbb{F}_{p}$ of perfect $\mathbb{F}_{p}$-schemes is stable under fibre product.

We show that any algebraic space \'{e}tale over a perfect field is perfect, which generalizes the result in \cite[Proposition 5.19]{Bertapellea}.
\begin{proposition}
Let $X$ be an algebraic space over $S$ and let $k$ be a perfect field of characteristic $p$. If $X$ is \'{e}tale over $k$, then $X$ is perfect.
\end{proposition}

More importantly, even our definitions of perfect algebraic spaces are completely separated from Frobenius morphisms, we recover a notion of Frobenius morphisms for algebraic spaces. Such a Frobenius morphism is called the \textit{algebraic Frobenius morphism}. The algebraic Frobenius morphisms satisfy some analogical properties of the absolute Frobenius morphisms of schemes. However, the following theorem is one of the most important results.
\begin{theorem}
Let $F$ be an algebraic space of characteristic $p$ over $S$. Then $F$ is perfect if and only if the algebraic Frobenius morphism $\Psi_{F}:F\rightarrow F$ of $F$ is an isomorphism.
\end{theorem}

By observing that the algebraic Frobenius morphism is representable, we introduce four types of perfect morphisms of algebraic spaces, called \textit{perfect}, \textit{quasi-perfect}, \textit{semiperfect}, and \textit{weakly perfect morphisms}. Let $\textrm{Hom}_{AP_{S}}^{P}(x,y)$, $\textrm{Hom}_{AP_{S}}^{Q}(x,y)$, $\textrm{Hom}_{AP_{S}}^{S}(x,y)$, and $\textrm{Hom}_{AP_{S}}^{W}(x,y)$ be the set of these perfect morphisms from $x$ to $y$, respectively. Then we have the following inclusions
\begin{align}
\textrm{Hom}_{AP_{S}}^{P}(x,y)\subset\textrm{Hom}_{AP_{S}}^{W}(x,y)\subset \textrm{Hom}_{AP_{S}}^{Q}(x,y)\subset \textrm{Hom}_{AP_{S}}^{S}(x,y).
\end{align}

Surprisingly, perfect morphisms are all stable under compositions and arbitrary base change.
\begin{proposition}[Composition]\label{T3}
Let $f:F\rightarrow G$ and $g:G\rightarrow H$ be morphisms of algebraic spaces over $S$.
\begin{itemize}
\item[(1)]
If $f,g$ are perfect, then the composition $g\circ f:F\rightarrow H$ is also perfect.
\item[(2)]
If $f,g$ are quasi-perfect, then the composition $g\circ f:F\rightarrow H$ is also quasi-perfect.
\item[(3)]
If $f,g$ are semiperfect, then the composition $g\circ f:F\rightarrow H$ is also semiperfect.
\item[(4)]
If $f,g$ are weakly perfect, then the composition $g\circ f:F\rightarrow H$ is also weakly perfect.
\end{itemize}
\end{proposition}

\begin{proposition}[Base change]
Let $F,G,H$ be algebraic spaces over $S$. Consider the following fibre product diagram
$$
\xymatrix{
  F\times_{H}G \ar[d]_{a'} \ar[r]^{\ \ b'} & F \ar[d]^{a} \\
  G \ar[r]^{b} & H   }
$$
If $a$ is perfect (resp. quasi-perfect, semiperfect, weakly perfect), then the base change $a'$ is also perfect (resp. quasi-perfect, semiperfect, weakly perfect).
\end{proposition}

More importantly, there are some equivalences between perfect algebraic spaces and perfect morphisms, which provide equivalent definitions for perfect algebraic spaces.
\begin{theorem}
Let $F$ be an algebraic space over $S$. Then
\begin{itemize}
\item[(1)]
$F$ is quasi-perfect if and only if there exists a perfect scheme $U\in{\rm{Ob}}(({\rm Sch}/S)_{fppf})$ such that every map $\varphi_{F}:h_{U}\rightarrow F$ is a quasi-perfect morphism of algebraic spaces over $S$. Equivalently, the diagonal $\Delta:F\rightarrow F\times F$ is quasi-perfect.
\item[(2)]
$F$ is semiperfect if and only if there exists $U\in{\rm{Ob}}(({\rm Sch}/S)_{fppf})$ such that every map $\varphi_{F}:h_{U}\rightarrow F$ is a semiperfect morphism of algebraic spaces over $S$. Equivalently, the diagonal $\Delta:F\rightarrow F\times F$ is semiperfect.
\item[(3)]
$F$ is strongly perfect if and only if for every perfect scheme $U\in{\rm{Ob}}(({\rm Sch}/S)_{fppf})$, the maps $\varphi_{F}:h_{U}\rightarrow F$ are weakly perfect. Equivalently, the diagonal $\Delta:F\rightarrow F\times F$ is weakly perfect.
\end{itemize}
\end{theorem}

Proposition \ref{T3} produces 12 categories of perfect algebraic spaces and perfect morphisms and 8 semicategories of perfect algebraic spaces and perfect morphisms. We construct three big commutative diagrams to characterize these categories and semicategories, see Figure \ref{F1}, Figure \ref{F2} and Figure \ref{F3}.

The subject on the perfection of algebraic spaces will be given in a second paper \cite{Liang}, where we define perfection of a given algebraic space via algebraic Frobenius. This enables us to pass between the usual and the perfect world.

\subsection{Outline}
We organize our paper as follows. First, we begin in \S\ref{BB1} by reviewing some commutative algebra surrounding perfect rings. Lemma \ref{C} will be utilized to prove Proposition \ref{A1} in \S\ref{B1}, while Lemma \ref{CC1} will be used to prove Proposition \ref{PP1} in \S\ref{B1}. Next, in \S\ref{B1}, we study perfect schemes and their perfections, and record some important results that will be useful in the sequel. Note that Proposition \ref{A1} is essential to deduce the results concerning fibre products in \S\ref{B2}.

All the material is put to use in \S\ref{B2}, and \S\ref{B2} is the starting point of the theory, where we give the definitions of four types of perfect algebraic spaces. In \S\ref{B3}, we first formalize the notion of characteristic of a given algebraic space. Then we deduce the algebraic Frobenius morphism and give its properties. Next, in \S\ref{B4}, we lay out the definitions of four types of perfect morphisms and study their relationships with perfect algebraic spaces in \S\ref{B2}. Then in \S\ref{B5}, we construct several categories and semicategories spanned by perfect algebraic spaces and perfect morphisms. Finally, in \S\ref{B6}, we briefly discuss the notion of perfect groupoids in algebraic spaces.

\subsection{Notations and conventions}
\begin{itemize}
  \item All rings are tacitly commutative with identity. $p,q$ are prime numbers and $\N=\{0,1,2,\dots\}$.
\end{itemize}
Meanwhile, we will follow some terminologies and notations in \cite{StackProject} as follows:
\begin{itemize}
  \item ${\rm Sch}$, the big category of schemes; ${\rm Sch}_{fppf}$, the big fppf site (see \cite[Tag021R]{StackProject}); and $\bf Sets$ the big category of sets.
  \item $S$ will always be a base scheme in a big fppf site ${\rm Sch}_{fppf}$. Then $({\rm Sch}/S)_{fppf}$ will denote the big fppf site of $S$ (see \cite[Tag021S]{StackProject}).
  \item Without explicitly mentioned, all schemes will be in ${\rm (Sch/S)}_{fppf}$.
\end{itemize}

Some foundational material for algebraic spaces can be found in \cite{StackProject,Art69c,Knutson}. We will stick with the definition of algebraic spaces in \cite[Tag025Y]{StackProject}.

By an \textit{algebraic space} over $S$, we mean a fppf sheaf $F$ on $({\rm Sch}/S)_{fppf}$ satisfying the usual axioms: its diagonal is representable, and it admits an \'{e}tale cover $h_{U}\rightarrow F$ for a scheme $U\in\textrm{Ob}(({\rm Sch}/S)_{fppf})$.

\section{Review of perfect rings and perfect schemes}\label{B}
In this section, we review some basic definitions and notions of perfect rings and perfect schemes. Meanwhile, we will state some important results that will be useful in the following sections. The material of this section can be found in \cite{Bertapellea, Greenberg, SGA3}.

\subsection{Perfect rings}\ \label{BB1}

Let $A$ be a ring of characteristic $p$. Recall that $A$ is \textit{perfect} if the Frobenius endomorphism $F_{A}:A\rightarrow A,\ a\mapsto a^{p}$ is an isomorphism.

\begin{lemma}\label{L3}
Let $f:A\rightarrow B$ be a homomorphism of rings. If $A$ has characteristic $p$ and $f$ is injective, then $B$ also has characteristic $p$.
\end{lemma}
\begin{proof}
Since $A$ has characteristic $p$, we have $f(p)=p=0$. Assume that $B$ has characteristic $p'<p$. Then we have $p'=f(p')=0$, which implies that $p'=0$ in $A$. This is a contradiction. Hence, $B$ has characteristic $p$ by the injective $f$.
\end{proof}

For the localizations at prime ideals of a perfect ring, we have the following lemma.
\begin{lemma}\label{C}
Let $A$ be a ring of characteristic $p$ and let $\mathfrak{p}\subset A$ be a prime ideal of $A$. If $A$ is perfect, then $A_{\mathfrak{p}}$ is also a perfect ring of characteristic $p$ and the Frobenius $\phi:A_{\mathfrak{p}}\rightarrow A_{\mathfrak{p}},\ x\mapsto x^{p}$ is an isomorphism.
\end{lemma}
\begin{proof}
Note that there is a canonical injection $A\rightarrow A_{\mathfrak{p}}$ which ensures that $A_{\mathfrak{p}}$ has characteristic $p$ due to Lemma \ref{L3}. For the injectivity of $\phi$. Let $a/b$ be an element in $A_{\mathfrak{p}}$ where $a\in A$ and $b\notin\mathfrak{p}$. If $(a/b)^{p}=a^{p}/b^{p}=0$, then we have $a^{p}=0$, which implies that $a=0$ since $A$ is perfect. Hence we have $a/b=0$ so that $\phi$ is injective.

To prove the surjectivity of $\phi$, we have to show that for every $x\in A_{\mathfrak{p}}$, there exists some $y\in A_{\mathfrak{p}}$ such that $y^{p}=x$. Let $a/b$ be an element in $A_{\mathfrak{p}}$ where $a\in A$ and $b\notin\mathfrak{p}$. Since $A$ is perfect, there exists some $a_{0}\in A$ such that $a_{0}^{p}=a$. Further, since $\mathfrak{p}$ is a prime ideal, $b^{p}\notin\mathfrak{p}$ implies that $b\notin\mathfrak{p}$. Thus the Frobenius $A\setminus\mathfrak{p}\rightarrow A\setminus\mathfrak{p}, b\mapsto b^{p}$ is bijective, which shows that for every $b\in A\setminus\mathfrak{p}$, there exists $b_{0}\in A\setminus\mathfrak{p}$ such that $b_{0}^{p}=b$. It follows that there is $a_{0}/b_{0}\in A_{\mathfrak{p}}$ such that $(a_{0}/b_{0})^{p}=a_{0}^{p}/b_{0}^{p}=a/b$. Thus, the Frobenius $\phi:A_{\mathfrak{p}}\rightarrow A_{\mathfrak{p}}$ is an isomorphism.
\end{proof}

Moreover, for the localization at some element in a perfect ring, we have the following lemma.
\begin{lemma}\label{CC1}
Let $A$ be a ring of characteristic $p$ and let $f\in A$. If $A$ is perfect, then $A_{f}$ is also a perfect ring of characteristic $p$, i.e. the Frobenius $\phi:A_{f}\rightarrow A_{f},\ x\mapsto x^{p}$ is an isomorphism.
\end{lemma}
\begin{proof}
Note that there is a canonical injection $A\rightarrow A_{f}$ which ensures that $A_{f}$ has characteristic $p$ due to Lemma \ref{L3}. It is obvious that $\phi$ is injective. For the surjectivity of $\phi$, we observe that for every $a/f^{m}\in A_{f}$ where $m\in\N_{>0}$, there exists some $f_{0}\in A$ with $f_{0}^{p}=f$ such that $(a^{1/p}f_{0}^{(p-1)m}/f^{m})^{p}=a/f^{m}$. Thus, the Frobenius $\phi:A_{f}\rightarrow A_{f}$ is an isomorphism.
\end{proof}

\subsection{Perfect schemes}\ \label{B1}

Recall that a scheme $X$ is said to be \textit{characteristic $p$}, if $p$ is zero in the structure sheaf $\mathscr{O}_{X}$, or equivalently, $X$ is an $\mathbb{F}_{p}$-scheme, i.e. $X$ is a scheme over a finite field $\mathbb{F}_{p}$.

Following the notation in \cite{SGA3}, the category of schemes in characteristic $p$ will be denoted by $\textrm{Sch}/\mathbb{F}_{p}$. And we denote the category of perfect schemes of characteristic $p$ by $\textrm{Perf}/\mathbb{F}_{p}$ following \cite{Bertapellea}. Clearly, $\textrm{Perf}/\mathbb{F}_{p}$ is a full subcategory of $\textrm{Sch}/\mathbb{F}_{p}$.

The \textit{absolute Frobenius} of a scheme $X$ is the morphism $\Phi_{X}:X\rightarrow X$ which is given by the identity on the underlying topological space, together with an endomorphism of sheaf $\Phi_{X}^{\sharp}:\mathscr{O}_{X}\rightarrow\mathscr{O}_{X},g\mapsto g^{p}$.
\begin{definition}[{\cite[\S5]{Bertapellea}}]
Let $X$ be a scheme of characteristic $p$. Then $X$ is said to be {\rm{perfect}} if the absolute Frobenius morphism $\Phi_{X}:X\rightarrow X$ of $X$ is an isomorphism.
\end{definition}

The following proposition characterizes the affine opens of perfect schemes.
\begin{proposition}\label{PP1}
Let $X$ be an $\mathbb{F}_{p}$-scheme. If $X$ is perfect, then every affine open $U\subset X$ is a perfect affine scheme. In particular, if $X={\rm{Spec}}(A)$ is a perfect affine scheme for some ring $A$, then every standard open set $D(f)\subset X,f\in A$ is a perfect affine subscheme of $X$.
\end{proposition}
\begin{proof}
Let $\Phi_{X}:X\rightarrow X$ be the isomorphic absolute Frobenius. Then for each affine open $U=\textrm{Spec}(R)\subset X$ for some ring $R$, the induced Frobenius $R\rightarrow R$ is an isomorphism. This implies that $\textrm{Spec}(R)$ is perfect. If $X=\textrm{Spec}(A)$ is affine, then for each standard open $D(f)\subset X,f\in A$, the induced Frobenius $A_{f}\rightarrow A_{f}$ is an isomorphism by Lemma \ref{CC1}. Thus, $D(f)$ is a perfect affine subscheme of $X$.
\end{proof}

Every scheme \'{e}tale over a perfect field of characteristic $p$ is perfect.
\begin{proposition}[{\cite[Proposition 5.19]{Bertapellea}}]\label{A13}
Let $k$ be a perfect field of characteristic $p$ and let $X$ be a $k$-scheme. If $X$ is \'{e}tale over $k$, then $X$ is perfect.
\end{proposition}

The category $\textrm{Perf}/\mathbb{F}_{p}$ of perfect schemes is stable under fibre products.
\begin{proposition}[{\cite[Remarks 5.18]{Bertapellea}}] \label{fibre product}
Let $k$ be a perfect field of characteristic $p$ and let $T$ be a perfect $k$-scheme. If $U,V$ are perfect $T$-schemes of characteristic $p$, then the fibre product $U\times_{T}V$ is also a perfect scheme of characteristic $p$.
\end{proposition}

We can decrease the conditions of \proref{fibre product} and obtain the following proposition.
\begin{proposition}\label{A1}
Let $U,V$ be perfect schemes of characteristic $p$ and let $f:U\rightarrow T,g:V\rightarrow T$ be a morphism of schemes. Then the fibre product $U\times_{T}V$ is a perfect scheme of characteristic $p$.
\end{proposition}
\begin{proof}
Assume that $U,V$ are perfect schemes of characteristic $p$ over $T$. It is easy to see that $U\times_{T}V$ always has characteristic $p$. To prove that the fibre product $U\times_{T}V$ is a perfect scheme, we should consider the structure sheaf of the fiber product $\mathscr{O}_{U\times_{T}V}$. Then it suffices to show that the Frobenius $\mathscr{O}_{U\times_{T} V}\rightarrow\mathscr{O}_{U\times_{T}V}$ is an isomorphism.

Take the stalk of $\mathscr{O}_{U\times_{T}V}$ at a point $(x,y,s,\mathfrak{p})$, where $x\in U$, $y\in V$, $s\in T$ with $f(x)=g(y)=s$, and $\mathfrak{p}\subset\mathscr{O}_{U,x}\otimes_{\mathscr{O}_{T,s}}\mathscr{O}_{V,y}$ is a prime ideal which restricts to maximal ideals in $\mathscr{O}_{U,x}$ and $\mathscr{O}_{V,y}$, i.e. $\mathfrak{p}\cap\mathscr{O}_{U,x}=\mathfrak{m}_{x},\mathfrak{p}\cap\mathscr{O}_{V,y}=\mathfrak{m}'_{y}$. Then we have
$$
\mathscr{O}_{U\times_{T}V,(x,y,s,\mathfrak{p})}=(\mathscr{O}_{U,x}\otimes_{\mathscr{O}_{T,s}}\mathscr{O}_{V,y})_{\mathfrak{p}}.
$$
%
Since $U,V$ are perfect schemes, the Frobenius morphisms $\mathscr{O}_{U}\rightarrow\mathscr{O}_{U}$ and $\mathscr{O}_{V}\rightarrow\mathscr{O}_{V}$ are isomorphisms, which implies that the Frobenius morphisms $\mathscr{O}_{U,x}\rightarrow\mathscr{O}_{U,x}$ and $\mathscr{O}_{V,y}\rightarrow\mathscr{O}_{V,y}$ for each $x\in U$ and $y\in V$ are isomorphisms by Lemma \ref{CC1}. Hence the Frobenius
$$
\mathscr{O}_{U,x}\otimes_{\mathscr{O}_{T,s}}\mathscr{O}_{V,y}\longrightarrow\mathscr{O}_{U,x}\otimes_{\mathscr{O}_{T,s}}\mathscr{O}_{V,y},\ g_{1}\otimes g_{2}\longmapsto g_{1}^{p}\otimes g_{2}^{p}=(g_{1}\otimes g_{2})^{p}
$$
is an isomorphism for $x\in U$ and $y\in V$. By Lemma \ref{C}, the Frobenius $\mathscr{O}_{U\times_{T} V,(x,y,s,\mathfrak{p})}\rightarrow\mathscr{O}_{U\times_{T}V,(x,y,s,\mathfrak{p})}$ is an isomorphism, which implies that the Frobenius $\mathscr{O}_{U\times_{T} V}\rightarrow\mathscr{O}_{U\times_{T}V}$ is an isomorphism.
\end{proof}

On the other hand, the absolute Frobenius defines a self-map of the identity functor on the category ${\rm Sch}/\mathbb{F}_{p}$.
\begin{lemma}[{\cite[Tag0CC7]{StackProject}}]\label{L4}
Let $f:X\rightarrow Y$ be a morphism of $\mathbb{F}_{p}$-schemes. Then the following diagram
$$
\xymatrix{
  X \ar[d]_{\Phi_{X}} \ar[r]^{f} & Y \ar[d]^{\Phi_{Y}} \\
  X \ar[r]^{f} & Y   }
$$
commutes, where $\Phi_X:X\to X$ (resp. $\Phi_Y:Y\to Y$) is the absolute Frobenius of $X$ (resp. $Y$).
\end{lemma}

Now, we work with schemes with some fixed base change. Let $X$ be a scheme over a base scheme $S$. Let $X^{(p)}:=X\times_{S,\Phi_{S}}S$ denote the base change of $X$ via the absolute Frobenius $\Phi_{S}:S\rightarrow S$. Here is the Cartesian diagram.
$$
\xymatrix{
  X^{(p)} \ar[d]_{\textrm{pr}_{S}}  \ar[r]^{\textrm{pr}_{X}} & X \ar[d]^{q} \\
  S \ar[r]^{\Phi_{S}} & S   }
$$
Then another commutative diagram
$$
\xymatrix{
  X \ar[d]_{q} \ar[r]^{\Phi_{X}} & X \ar[d]^{q} \\
  S \ar[r]^{\Phi_{S}} & S   }
$$
(where $\Phi_{X}$ is the absolute Frobenius of $X$) yields a unique $S$-morphism $\Phi_{X/S}:X\rightarrow X^{(p)}$ that makes the diagram
$$
\xymatrix{
  X \ar@/_/[ddr]_{q} \ar@/^/[drr]^{\Phi_{X}}
    \ar@{.>}[dr]|-{\Phi_{X/S}}                   \\
   & X^{(p)} \ar[d]^{} \ar[r]_{\textrm{pr}_{X}}
                      & X \ar[d]_{q}    \\
   & S \ar[r]^{\Phi_{S}}     & S               }
$$
commute.
\begin{definition}[\cite{SGA3}, VII$_{A}$, \S4]
Consider the situation above. The unique $S$-morphism $\Phi_{X/S}:X\rightarrow X^{(p)}$ is called the relative Frobenius morphism of $X$.
\end{definition}

More generally, for every $n\in\N$, let $X^{(p_n^{n})}:=X\times_{S,\Phi_{S}^{n}}S_{\Phi_{S}^{n}}$, where $\Phi_{S}^{n}$ means the composite of $\Phi_{S}$ for $n$ times and $S_{\Phi_{S}^{n}}$ is regarded as an $S$-scheme via $\Phi_{S}^{n}$. Then we have the Cartesian diagram
$$
\xymatrix{
  X^{(p_n^{n})} \ar[d]_{\textrm{pr}_{S}} \ar[r]^{\textrm{pr}_{X}} & X \ar[d]^{q} \\
  S \ar[r]^{\Phi_{S}^{n}} & S   }
$$

If the base scheme $S$ is perfect, then the absolute Frobenius $\Phi_{S}:S\to S$ is an isomorphism. This enables us to consider its inverse $\Phi_{S}^{-1}$. Hence, we can extend $n$ above to arbitrary integers. Let $X^{(p^{n})}:=X\times_{S,\Phi_{S}^{n}}S$ for $n\in\Z$. Then we have a canonical isomorphism
$$
(X^{(p^{n})})^{(p)}\cong X^{(p^{n+1})}.
$$

Note that $X\mapsto X^{(p)}$ is a functor. In fact, it is the base change functor of the absolute Frobenius $\Phi_{S}:S\rightarrow S$. Thus, we have the following lemma analogous to Lemma \ref{L4}.
\begin{lemma}[{\cite[Tag0CCA]{StackProject}}]\label{L5}
Let $f:X\rightarrow Y$ is a morphism of $\mathbb{F}_{p}$-schemes over $S$. Then the following diagram
$$
\xymatrix{
  X \ar[d]_{f} \ar[r]^{\Phi_{X/S}} & X^{(p)} \ar[d]^{f^{(p)}} \\
  Y \ar[r]^{\Phi_{Y/S}} & Y^{(p)}   }
$$
commutes.
\end{lemma}


\section{Perfect algebraic spaces}\label{B2}
In this section, we extend the notion of perfect schemes to the setting of algebraic spaces. We will stick with the terminologies and notations of \cite[Tag025Y]{StackProject} for algebraic spaces. From now on, $S$ will always be a fixed base scheme contained in the big fppf site ${\rm Sch}_{fppf}$.

\begin{definition}\label{D1}
Let $F$ be an algebraic space over $S$.
\begin{itemize}
\item[(1)]
We say that $F$ is perfect if there exists a surjective \'{e}tale map $h_{U}\rightarrow F$, where $U\in{\rm{Ob}}(({\rm Sch}/S)_{fppf})$ is a perfect scheme.
\item[(2)]
We say that $F$ is quasi-perfect if there exist perfect schemes $U,V\in{\rm{Ob}}(({\rm Sch}/S)_{fppf})$ such that for any $a\in F(U),b\in F(V)$, the functor $h_{V}\times_{a,F,b}h_{U}$ is represented by a perfect scheme $W\in{\rm{Ob}}(({\rm Sch}/S)_{fppf})$.
\item[(3)]
We say that $F$ is semiperfect if there exist $U,V\in{\rm{Ob}}(({\rm Sch}/S)_{fppf})$ such that for any $a\in F(U),b\in F(V)$, the functor $h_{V}\times_{a,F,b}h_{U}$ is represented by a perfect scheme $W\in{\rm{Ob}}(({\rm Sch}/S)_{fppf})$.
\item[(4)]
We say that $F$ is strongly perfect if for every perfect schemes $U,V\in{\rm{Ob}}(({\rm Sch}/S)_{fppf})$ and any $a\in F(U),b\in F(V)$, the fibre product $h_{U}\times_{a,F,b}h_{V}$ is represented by a perfect scheme $W\in{\rm{Ob}}(({\rm Sch}/S)_{fppf})$.
\end{itemize}
\end{definition}
\begin{remark}
Note that in (2), (3) and (4) above, one might take $U=V$.
\end{remark}

The following lemma shows that \defref{D1} does generalize the notion of perfect scheme.

\begin{lemma}\ \label{A3}

{\rm (1)} A perfect scheme is a perfect algebraic space. More precisely, given a perfect scheme $T\in{\rm{Ob}}(({\rm Sch}/S)_{fppf})$, the representable functor $h_{T}$ is a perfect algebraic space.

{\rm (2)} A perfect scheme is a quasi-perfect algebraic space. More precisely, given a perfect scheme $T\in{\rm{Ob}}(({\rm Sch}/S)_{fppf})$, the representable functor $h_{T}$ is a quasi-perfect algebraic space.

{\rm (3)} A perfect scheme is a semiperfect algebraic space. More precisely, given a perfect scheme $T\in{\rm{Ob}}(({\rm Sch}/S)_{fppf})$, the representable functor $h_{T}$ is a semiperfect algebraic space.

{\rm (4)} A perfect scheme is a strongly perfect algebraic space. More precisely, given a perfect scheme $T\in{\rm{Ob}}(({\rm Sch}/S)_{fppf})$, the representable functor $h_{T}$ is a strongly perfect algebraic space.

In other words, every algebraic space that is represented by a perfect scheme is perfect. And every representable algebraic space is quasi-perfect, semiperfect, and strongly perfect.
\end{lemma}
\begin{proof}
By \cite[Tag025Z]{StackProject}, the functor $h_{T}$ is an algebraic space. Since $T$ is a perfect scheme and the identity map $h_{T}\rightarrow h_{T}$ is surjective \'{e}tale, $h_{T}$ is perfect. This proves (1).

Let $U,V\in{\rm{Ob}}(({\rm Sch}/S)_{fppf})$ be perfect schemes and $\xi\in h_{T}(U),\xi'\in h_{T}(V)$. Then we have $h_{V}\times_{h_{T}}h_{U}=h_{V\times_{T}U}$. By Proposition \ref{A1}, $V\times_{T}U$ is perfect. Hence, $h_{T}$ is strongly perfect. This proves (2), (3) and (4).
\end{proof}



Clearly, the category $\textrm{Perf}_{AP_{S}}$ of perfect algebraic spaces over $S$ and the category $\textrm{SPerf}_{AP_{S}}$ of semiperfect algebraic spaces over $S$ are full subcategories of the category $AP_{S}$ of algebraic spaces over $S$. And the category $\textrm{QPerf}_{AP_{S}}$ of quasi-perfect algebraic spaces over $S$ is a full subcategory of $\textrm{SPerf}_{AP_{S}}$. Moreover, the category $\textrm{StPerf}_{AP_{S}}$ of strongly perfect algebraic spaces over $S$ is the full subcategory of $\textrm{QPerf}_{AP_{S}}$ and $\textrm{SPerf}_{AP_{S}}$.

Hence, we have the following full embedding
\begin{align}
\textrm{Perf}_{AP_{S}}\subset AP_{S},\label{A22}
\end{align}
together with a string of full embeddings
\begin{align}
\textrm{StPerf}_{AP_{S}}\subset \textrm{QPerf}_{AP_{S}}\subset \textrm{SPerf}_{AP_{S}}\subset AP_{S}.  \label{I1}
\end{align}

Let $\widehat{{\rm{Perf}}_{AP_{S}}}$ (resp. $\widehat{{\rm{StPerf}}_{AP_{S}}}$, $\widehat{{\rm{QPerf}}_{AP_{S}}}$, $\widehat{{\rm{SPerf}}_{AP_{S}}}$) be the subcategory of ${\rm{Perf}}_{AP_{S}}$ (resp. ${\rm{StPerf}}_{AP_{S}}$, ${\rm{QPerf}}_{AP_{S}}$, ${\rm{SPerf}}_{AP_{S}}$) of representable perfect (resp. strongly perfect, quasi-perfect, semiperfect) algebraic spaces over $S$.
%
%
%
%
%
%
It follows from Lemma \ref{A3} that we have the following full embeddings
\begin{align}
\widehat{{\rm{Perf}}_{AP_{S}}}\subset\widehat{{\rm{StPerf}}_{AP_{S}}}\subset\widehat{{\rm{QPerf}}_{AP_{S}}}\subset\widehat{{\rm{SPerf}}_{AP_{S}}}.
\end{align}

In the following, we give an interesting lemma of homological algebra that transfers fibre products. It will be useful in the sequel.
\begin{lemma}\label{A5}
Let $\mathcal{C}$ be a category and let $\alpha:F\rightarrow H$ and $\beta:G\rightarrow H$ be morphisms in $\mathcal{C}$. Assume that the fibre product $F\times_{H}G$ exists and we are given a monomorphism $\gamma:H\hookrightarrow W$. Then we have $F\times_{H}G=F\times_{W}G$. More precisely, the fibre product of $\alpha$ and $\beta$ is the fibre product of $\gamma\circ \alpha$ and $\gamma\circ \beta$.
\end{lemma}
\begin{proof}
Consider the following commutative pull-back diagram
$$
\xymatrix{
  Q \ar@/_/[ddr]_{f_{2}} \ar@/^/[drr]^{f_{1}}
    \ar@{.>}[dr]|-{\xi}                   \\
   & F\times_{H}G \ar[d]^{p_{2}} \ar[r]_{p_{1}}
                      & F \ar[d]_{\alpha}    \\
   & G \ar[r]^{\beta}     & H               }
$$
It is easy to see that the monomorphism $\gamma:H\hookrightarrow W$ induces a commutative diagram
$$
\xymatrix{
  F\times_{H}G \ar[d]_{p_{2}} \ar[r]^{p_{1}} & F \ar[d]^{\gamma\circ \alpha} \\
  G \ar[r]^{\gamma\circ \beta} & W   }
$$
We claim that $F\times_{H}G$ is also the fibre product of $\gamma\circ \alpha$ and $\gamma\circ \beta$ indeedly. For every $Q\in\textrm{Ob}(\mathcal{C})$ with $f_{1}'\in\textrm{Hom}_{\mathcal{C}}(Q,F)$ and $f_{2}'\in\textrm{Hom}_{\mathcal{C}}(Q,G)$ such that $\gamma\circ \alpha \circ  f_{1}'=\gamma \circ \alpha \circ  f_{2}'$, we have $\alpha \circ  f_{1}'=\alpha \circ  f_{2}'$ since $\gamma$ is a monomorphism. Thus, there exists a unique morphism $\xi':Q\rightarrow F\times_{H}G$ such that $p_{1}\circ  \xi'=f_{1}'$ and $p_{2}\circ  \xi'=f_{2}'$. Therefore, $F\times_{H}G=F\times_{W}G$.
\end{proof}

Next, we will show that in certain cases perfect algebraic spaces are stable under fibre products.
\begin{proposition}\label{A4}
Let $F,G$ be algebraic spaces over $S$. Let $H$ be a sheaf on $({\rm Sch}/S)_{fppf}$ whose diagonal morphism is representable. Let $F\rightarrow H,G\rightarrow H$ be maps of sheaves. Suppose that there exist perfect schemes $U,V\in\ObSchS$ and surjective \'{e}tale maps $h_{U}\rightarrow F,h_{V}\rightarrow G$ such that $h_{U}\times_{H}h_{V}$ is represented by a perfect scheme. If $F,G$ are perfect, then the fibre product $F\times_{H}G$ is a perfect algebraic space.
\end{proposition}
\begin{proof}
By virtue of \cite[Tag04T9]{StackProject}, $F\times_{H}G$ is an algebraic space. Note that the morphism $h_{U}\times_{H}h_{V}\rightarrow F\times_{H}G$ as the composition of base changes of $h_{U}\times_{H}h_{V}\rightarrow h_{U}\times_{H}G$ and $h_{U}\times_{H}G\rightarrow F\times_{H}G$ is surjective \'{e}tale. By assumption, there exists perfect scheme $W\in\ObSchS$ such that $h_{U}\times_{H}h_{V}\simeq h_{W}$. Thus, $F\times_{H}G$ is perfect.
\end{proof}

Applying the \proref{A4}, we obtain the following proposition.
\begin{proposition}\label{A7}
Let $F\rightarrow H$ and $G\rightarrow H$ be morphisms of algebraic spaces over $S$. Suppose that $H$ is quasi-perfect such that there exist perfect schemes $U,V\in\ObSchS$ and surjective \'{e}tale maps $h_{U}\rightarrow F,h_{V}\rightarrow G$ making $h_{U}\times_{H}h_{V}$ represented by a perfect scheme. If $F,G$ are perfect, then the fibre product $F\times_{H}G$ is a perfect algebraic space. If $H$ is also a perfect algebraic space, then $F\times_{H}G$ is a fibre product in the category ${\rm{Perf}}_{AP_{S}}$ of perfect algebraic spaces over $S$.
\end{proposition}
\begin{proof}
The proof follows from the stronger Proposition \ref{A4} that $F\times_{H}G$ is a perfect algebraic space. Furthermore, by \cite[Tag04T9]{StackProject}, $F\times_{H}G$ is a fibre product in the category of algebraic spaces. Then it is clear that if $H$ is also perfect, then $F\times_{H}G$ comes as a fibre product in ${\rm{Perf}}_{AP_{S}}$ since that is a full subcategory of the category of algebraic spaces.
\end{proof}

Similarly, one might be desirable to show that the same property holds for the category $\textrm{QPerf}_{AP_{S}}$ of quasi-perfect algebraic spaces, the category $\textrm{SPerf}_{AP_{S}}$ of semiperfect algebraic spaces, and the category $\textrm{StPerf}_{AP_{S}}$ of strongly perfect algebraic spaces.

\begin{proposition}\label{P6}
Let $F,G$ be algebraic spaces over $S$ and let $H$ be a sheaf on $({\rm Sch}/S)_{fppf}$ whose diagonal morphism is representable. Assume that $F\rightarrow H,G\rightarrow H$ are morphisms of sheaves.
\begin{enumerate}
\item If $F,G$ are semiperfect, then the fibre product $F\times_{H}G$ is also a semiperfect algebraic space.
\item If $F,G$ are quasi-perfect, then the fibre product $F\times_{H}G$ is also a quasi-perfect algebraic space.
\item If $F,G$ are strongly perfect, then the fibre product $F\times_{H}G$ is also a strongly perfect algebraic space.
\end{enumerate}
\end{proposition}
We defer the proof of \proref{P6} in the sequel until \S\ref{B4}. Applying the \proref{P6}, we obtain the desired results.
\begin{proposition}
Let $F\rightarrow H$ and $G\rightarrow H$ be morphisms of algebraic spaces over $S$.
\begin{enumerate}
\item If $F,G$ are semiperfect, then the fibre product $F\times_{H}G$ is also a semiperfect algebraic space. Furthermore, if $H$ is also a semiperfect algebraic space, then $F\times_{H}G$ is a fibre product in the category ${\rm{SPerf}}_{AP_{S}}$ of semiperfect algebraic spaces over $S$.
\item If $F,G$ are quasi-perfect, then the fibre product $F\times_{H}G$ is also a quasi-perfect algebraic space. Furthermore, if $H$ is also a quasi-perfect algebraic space, then $F\times_{H}G$ is a fibre product in the category ${\rm{QPerf}}_{AP_{S}}$ of quasi-perfect algebraic spaces over $S$.
\item If $F,G$ are strongly perfect, then the fibre product $F\times_{H}G$ is also a strongly perfect algebraic space. Furthermore, if $H$ is also a strongly perfect algebraic space, then $F\times_{H}G$ is a fibre product in the category ${\rm{StPerf}}_{AP_{S}}$ of strongly perfect algebraic spaces over $S$.
\end{enumerate}
\end{proposition}
\begin{proof}
The proof of (1) follows from the stronger Proposition \ref{P6} that $F\times_{H}G$ is a semiperfect algebraic space. Then it is clear that if $H$ is also semiperfect, $F\times_{H}G$ comes as a fibre product in ${\rm{SPerf}}_{AP_{S}}$ since that is a full subcategory of the category $AP_{S}$ of algebraic spaces over $S$. Similarly, we can obtain (2) and (3).
\end{proof}

Let $k$ be a fixed perfect field of characteristic $p$. The following proposition generalizes Proposition \ref{A13}.
\begin{proposition}
Let $X$ be an algebraic space over $S$. If $X$ is \'{e}tale over $k$, then $X$ is perfect.
\end{proposition}
\begin{proof}
Choose a surjective \'{e}tale map $h_{U}\rightarrow X$ for $U\in\Ob(Sch/k)$. This gives rise to an \'{e}tale morphism $U\rightarrow \textrm{Spec}(k)$ as the composition of \'{e}tale morphisms. Then it follows from Proposition \ref{A13} that $U$ is perfect. Thus, $X$ is perfect.
\end{proof}

For the relation between \'{e}tale morphisms and perfect algebraic spaces, we will make use of the result in Proposition \ref{A13}, and give the following proposition which specifies that an \'{e}tale morphism in some cases would imply semiperfectness or quasi-perfectness.
\begin{proposition}
Let $f:F\rightarrow G$ be an \'{e}tale morphism of algebraic spaces over $S$. If the functor $h_{U}\times_{f\circ a,G,b}h_{V}$ is represented by an affine scheme ${\rm{Spec}}(k)$ for some $U,V\in{\rm{Ob}}(({\rm Sch}/S)_{fppf})$ and any $a\in F(U),b\in G(V)$, then $F$ is semiperfect. In particular, if $U,V\in{\rm{Ob}}(({\rm Sch}/S)_{fppf})$ are perfect schemes, then $F$ is quasi-perfect.
\end{proposition}
\begin{proof}
Let $U,V\in{\rm{Ob}}(({\rm Sch}/S)_{fppf})$ and $x\in F(U),y\in F(V),y'\in G(V)$. Choose isomorphisms $h_{W}\simeq h_{U}\times_{x,F,y}h_{V}$ and $h_{W'}\simeq h_{U}\times_{f\circ x,G,y'}h_{V}$ for $W,W'\in{\rm{Ob}}(({\rm Sch}/S)_{fppf})$. Then the compositions
$$
h_{W}\rightarrow h_{U}\times_{F}h_{V}\rightarrow h_{U}\ {\rm and}\ h_{W'}\rightarrow h_{U}\times_{G}h_{V}\rightarrow h_{U}
$$
come from a unique pair of morphisms of schemes $W\rightarrow U$ and $W'\rightarrow U$ by Yoneda lemma.
For every \'{e}tale map $\varphi:h_{V}\rightarrow F$, since $f$ is an \'{e}tale morphism of algebraic spaces over $S$, $f\circ\varphi:h_{V}\rightarrow G$ is also \'{e}tale. In other words, if the morphism of schemes $W\rightarrow U$ is \'{e}tale, then the morphism of schemes $W'\rightarrow U$ is also \'{e}tale. By \cite[Tag02GW]{StackProject}, if $g:W\rightarrow W'$ is a morphism of schemes over $U$, then $g$ is \'{e}tale. It follows from Proposition \ref{A13} that if $W'$ is the affine scheme $\textrm{Spec}(k)$, then $W$ as an \'{e}tale $k$-scheme is perfect. Hence, the functor $h_{U}\times_{F}h_{V}$ is represented by a perfect scheme so that $F$ is semiperfect. And since we may choose $U,V\in{\rm{Ob}}(({\rm Sch}/S)_{fppf})$ to be perfect, $F$ is also quasi-perfect.
\end{proof}

The following proposition shows that one can transfer perfectness from an algebraic space to another.
\begin{proposition}
Let $G$ be an algebraic space over $S$, $F$ be a sheaf on $({\rm Sch}/S)_{fppf}$, and $G\rightarrow F$ be a representable morphism of functors which is surjective \'{e}tale. If $G$ is perfect, then $F$ is a perfect algebraic space.
\end{proposition}
\begin{proof}
It follows from \cite[Tag0BGR]{StackProject} that $F$ is an algebraic space. Since $G$ is perfect, we have a surjective \'{e}tale morphism $h_{U}\rightarrow G$ for a perfect scheme $U\in\textrm{Ob}(({\rm Sch}/S)_{fppf})$. Then as in \cite[Tag0BGR]{StackProject}, the composition $h_{U}\rightarrow G\rightarrow F$ is surjective \'{e}tale. Thus, $F$ is perfect.
\end{proof}

From \cite[Tag0BGQ]{StackProject}, we can weaken the conditions such that a sheaf to be a perfect algebraic space.
\begin{proposition}
Let $F$ be a sheaf on $({\rm Sch}/S)_{fppf}$ such that there is a perfect scheme $U\in {\rm{Ob}}(({\rm Sch}/S)_{fppf})$ and a representable map $h_{U}\rightarrow F$ which is surjective \'{e}tale. Then $F$ is a perfect algebraic space.
\end{proposition}

\section{Algebraic Frobenius morphisms}\label{B3}
In this section, we introduce the notion of algebraic Frobenius morphisms so that we can describe a perfect algebraic space $F$ in terms of the endomorphism $F\rightarrow F$.

\subsection{Algebraic Frobenius morphisms}\

As the usual cases, the Frobenius morphism of an algebraic space $F$ makes sense only when $F$ has characteristic $p$. So we need to formalize the characteristic of a given algebraic space.
\begin{definition}
Let $F$ be an algebraic space over $S$.
\begin{itemize}
\item[(1)]
We say that $F$ has characteristic $p$ if $F$ is nonempty and there exists a surjective \'{e}tale map $h_{U}\rightarrow F$ where $U$ has characteristic $p$.
\item[(2)]
$F$ is said to have characteristic $0$ if $F$ is nonempty and for every surjective \'{e}tale map $h_{U}\rightarrow F$ where $U\in {\rm{Ob}}(({\rm Sch}/S)_{fppf})$, $U$ does not have characteristic $p$.
\end{itemize}
We will use ${\rm{char}}(F)$ to indicate the characteristic of $F$.
\end{definition}
\begin{remark}
If $F$ is a scheme, then we say $F$ \textit{has weak characteristic} $p$ or $0$ to distinguish it with the usual case. Clearly, when $F$ is an $\mathbb{F}_{p}$-scheme, $F$ has weak characteristic $p$.
\end{remark}

In fact, the characteristic of an algebraic space is independent of the choice of \'{e}tale atlases.
\begin{lemma}\label{LL2}
Let $F$ be a nonempty algebraic space over $S$. Suppose that there are two surjective \'{e}tale maps $h_{U}\rightarrow F$ and $h_{V}\rightarrow F$ where $U,V\in {\rm{Ob}}(({\rm Sch}/S)_{fppf})$ have characteristics $p,q$, respectively. Then $p=q$.
\end{lemma}
\begin{proof}
Consider the fibre product $h_{U}\times_{F}h_{V}\simeq h_{W}$ for some $W\in {\rm{Ob}}(({\rm Sch}/S)_{fppf})$. The projection $h_{W}\simeq h_{U}\times_{F}h_{V}\rightarrow h_{V}$ gives rise to a morphism of schemes $W\rightarrow V$ that makes $W$ to be an $\mathbb{F}_{q}$-scheme. Similarly, there is another morphism of schemes $W\rightarrow U$ that makes $W$ to be an $\mathbb{F}_{p}$-scheme. This shows that $p=q$ as desired.
\end{proof}

For the morphisms of algebraic spaces with different characteristics, we have the following lemma.
\begin{lemma}
Let $F\rightarrow G$ be a morphism of algebraic spaces over $S$. Suppose that $F,G$ have characteristics $p,q$, respectively. Then $p=q$.
\end{lemma}
\begin{proof}
Let $h_{V}\rightarrow G$ be a surjective \'{e}tale map where $V\in {\rm{Ob}}(({\rm Sch}/S)_{fppf})$ has characteristic $q$. By \cite[Tag02X1]{StackProject}, there exist $U\in {\rm{Ob}}(({\rm Sch}/S)_{fppf})$ and a commutative diagram
$$
\xymatrix{
  U \ar[d]_{} \ar[r]^{} & V \ar[d]^{} \\
  F \ar[r]^{} & G   }
$$
where $U\rightarrow F$ is surjective \'{e}tale. This shows that $U$ is an $\mathbb{F}_{q}$-scheme. Now, choose a surjective \'{e}tale map $h_{W}\rightarrow F$ where $W\in {\rm{Ob}}(({\rm Sch}/S)_{fppf})$ has characteristic $p$. Then $p=q$ following Lemma \ref{LL2}.
\end{proof}

The following lemma ensures the existence of certain morphisms of sheaves. It will be useful in \proref{A11} that gives rise to the notion of algebraic Frobenius morphism.
\begin{lemma}\label{A10}
Let $\mathcal{C}$ be a site and $F,F',G,G':\mathcal{C}^{opp}\rightarrow {\bf Sets}$ be sheaves of sets on $\calc$. Let $f:G\rightarrow F$, $g:G'\rightarrow F'$, and $\varphi:G\rightarrow G'$ be morphisms of sheaves. If $f$ is surjective, then there is a canonical map $h:F\rightarrow F'$ such that $h\circ f=g\circ\varphi$. In other words, there exists a unique morphism $h:F\rightarrow F'$ of sheaves making the diagram
$$
\xymatrix{
  G \ar[d]_{\varphi} \ar[r]^{f} & F \ar@{-->}[d]^{h} \\
  G' \ar[r]^{g} & F'   }
$$
commute.
\end{lemma}
\begin{proof}
Let $X,Y\in\textrm{Ob}(\mathcal{C})$. Then we have a solid diagram
$$
\xymatrix{
  G(X) \ar[d]_{\varphi_{X}} \ar[r]^{f_{X}} & F(X) \ar@{-->}[d]^{h_{X}} \\
  G'(X) \ar[r]^{g_{X}} & F'(X)   }
$$
in ${\bf Sets}$. The function $h_{X}:F(X)\rightarrow F'(X)$ is given by
$$
\begin{cases}
h_{X}(f_{X}(x))=g_{X}(\varphi_{X}(x)), & \textrm{for all} \ x\in G(X);\\
h_{X}(y)=x_{0},  & \textrm{for all} \ y\in F(X)\setminus \textrm{Im}(f_{X})\textrm{ and some }x_{0}\in F'(X).
\end{cases}
$$
It is obvious that $h_{X}$ is well-defined. Hence, $h_{X}$ fits into the commutative dotted square above. Let $u:Y\rightarrow X$ be a morphism in $\calc$. Then we have
$$
F'(u)(h_{X}(f_{X}(x')))=g_{Y}\varphi_{Y}G(u)(x')=h_{Y}f_{Y}G(u)(x')=h_{Y}(F(u)(f_{X}(x')))
$$
for all $x'\in G(X)$. This gives rise to a commutative diagram
$$
\xymatrix{
  G(X) \ar[d]_{G(u)} \ar[r]^{f_{X}} & F(X) \ar[d]_{F(u)} \ar[r]^{h_{X}} & F'(X) \ar[d]^{F'(u)} \\
  G(Y) \ar[r]^{f_{Y}} & F(Y) \ar[r]^{h_{Y}} & F'(Y)   }
$$
which implies that $F'(u)h_{X}=h_{Y}F(u)$, i.e. the diagram
$$
\xymatrix{
  F(X) \ar[d]_{F(u)} \ar[r]^{h_{X}} & F'(X) \ar[d]^{F'(u)} \\
  F(Y) \ar[r]^{h_{Y}} & F'(Y)   }
$$
commutes. Thus, $h$ is a morphism of sheaves. Let $h':F\rightarrow F'$ be another morphism of sheaves such that $h'\circ f=h\circ f=g\circ \varphi$. Then $h'=h$ following the assumption that $f$ is surjective.
\end{proof}

\begin{proposition}\label{A11}
Let $F$ be an algebraic space of characteristic $p$ over $S$ and let $f:h_{U}\rightarrow F$ be a surjective \'{e}tale map for $U\in{\rm{Ob}}(({\rm Sch}/S)_{fppf})$ of characteristic $p$. Then the absolute Frobenius morphism $\Phi_{U}:U\rightarrow U$ of $U$ induces a canonical map $\Psi_{F}^{U}:F\rightarrow F$ such that $\Psi_{F}^{U}\circ f=f\circ h(\Phi_{U})$. In other words, there exists a unique morphism $\Psi_{F}^{U}:F\rightarrow F$ that fits into the following commutative dotted diagram
$$
\xymatrix{
  h_{U} \ar[d]_{h(\Phi_{U})} \ar[r]^{f} & F \ar@{-->}[d]^{\Psi_{F}^{U}} \\
  h_{U} \ar[r]^{f} & F   }
$$
\end{proposition}
\begin{proof}
It follows from \cite[Tag05VM]{StackProject} that the surjective \'{e}tale map $f:h_{U}\rightarrow F$ is surjective as a map of sheaves. Thus, the proposition follows directly from Lemma \ref{A10}.
\end{proof}

Now, we can make the definition of the algebraic Frobenius morphisms.
\begin{definition}
Let $F$ be an algebraic space of characteristic $p$ over $S$ and let $f:h_{U}\rightarrow F$ be a surjective \'{e}tale map where $U\in{\rm{Ob}}(({\rm Sch}/S)_{fppf})$ has characteristic $p$. The algebraic Frobenius morphism of $F$ with respect to $U$ is the canonical morphism $\Psi_{F}^{U}:F\rightarrow F$ as in Proposition \ref{A11}.
\end{definition}
\begin{remark}
By abuse of notation, we suppress the superscript of the algebraic Frobenius morphism and simply denote it by $\Psi_{F}$. When we speak of the \textit{algebraic Frobenius morphism} $\Psi_{F}$ of $F$, it is understood that $\Psi_{F}$ is one of the algebraic Frobenius morphisms of $F$ with respect to some scheme.
\end{remark}

If an algebraic space $F$ has characteristic 0, then the algebraic Frobenius of $F$ will not make sense. The algebraic Frobenius morphism of $F$ is unique if and only if ${\rm{char}}(F)\neq0$ and $F$ has only one \'{e}tale atlas from a scheme in characteristic $p$. Moreover, if ${\rm{char}}(F)\neq0$, then the number of algebraic Frobenius morphisms of $F$ is equal to the number of \'{e}tale atlases of $F$ from schemes in characteristic $p$. Thus, we have a bijection of sets


\begin{equation*}
\begin{matrix}  \{\text{The set of all \'{e}tale atlases of }F

\\  \text{ from schemes in characteristic }p\}

\end{matrix} \longleftrightarrow \begin{matrix}  \{\text{The set of all algebraic Frobenius}

\\   \text{morphisms of }F\}

\end{matrix}
\end{equation*}

In other words, each algebraic Frobenius morphism $\Psi_{F}:F\rightarrow F$ of $F$ corresponds to an \'{e}tale atlas $\varphi_{F}:h_{U}\rightarrow F$ from a scheme $U$ in characteristic $p$. Such an algebraic Frobenius morphism is called the \textit{algebraic Frobenius morphism of $F$ with respect to} $\varphi_{F}$. When we speak of a perfect algebraic space $F$, then the \textit{algebraic Frobenius morphism of} $F$ means the algebraic Frobenius morphism $\Psi_{F}$ of $F$ with respect to one of its \'{e}tale atlases from perfect schemes.

The algebraic Frobenius induces a self-map of the identity functor on the category $AP_{S}$ of algebraic spaces over $S$.
\begin{proposition}\label{A6}
Let $f:F\rightarrow G$ be a morphism of algebraic spaces in characteristic $p$ over $S$ and let $\Psi_{F}:F\rightarrow F,\Psi_{G}:G\rightarrow G$ be the algebraic Frobenius morphisms of $F,G$. Then the diagram
$$
\xymatrix{
  F \ar[d]_{\Psi_{F}} \ar[r]^{f} & G \ar[d]^{\Psi_{G}} \\
  F \ar[r]^{f} & G   }
$$
is commutative.
\end{proposition}
\begin{proof}
Choose an \'{e}tale cover $\varphi_{F}:h_{U}\rightarrow F$ for $U\in{\rm{Ob}}(({\rm Sch}/S)_{fppf})$ of characteristic $p$. Then it follows from \cite[Tag02X1]{StackProject} that there exists $V\in{\rm{Ob}}(({\rm Sch}/S)_{fppf})$ and a commutative diagram
$$
\xymatrix{
  h_{U} \ar[d]_{\varphi} \ar[r]^{\varphi_{F}} & F \ar[d]^{f} \\
  h_{V} \ar[r]^{\varphi_{G}} & G   }
$$
where $\varphi_{G}$ is surjective \'{e}tale. By Yoneda lemma, $\varphi$ comes from a unique morphism of schemes $g:U\rightarrow V$. Thus, we have $f\circ \varphi_{F}=\varphi_{G} \circ h(g)$.

Now, consider the commutative diagrams
$$
\xymatrix{
  h_{U} \ar[d]_{h(\Phi_{U})} \ar[r]^{\varphi_{F}} & F \ar@{-->}[d]^{\Psi_{F}} \\
  h_{U} \ar[r]^{\varphi_{F}} & F   }
$$
and
$$
\xymatrix{
  h_{V} \ar[d]_{h(\Phi_{V})} \ar[r]^{\varphi_{G}} & G \ar@{-->}[d]^{\Psi_{G}} \\
  h_{V} \ar[r]^{\varphi_{G}} & G   }
$$
i.e. we have $\Psi_{F}\circ \varphi_{F}=\varphi_{F}\circ h(\Phi_{U})$ and $\Psi_{G}\circ \varphi_{G}=\varphi_{G}\circ h(\Phi_{V})$. Then composing each side of the first equality by $f$ gives
$$
f\circ \Psi_{F}\circ \varphi_{F}=f\circ \varphi_{F}\circ h(\Phi_{U})=\varphi_{G}\circ h(g)\circ h(\Phi_{U})=\varphi_{G}\circ h(g\Phi_{U}).
$$
By Lemma \ref{L4}, we have $g\Phi_{U}=\Phi_{V}g$. Hence,
$$
\varphi_{G}h(g\Phi_{U})=\varphi_{G}h(\Phi_{V}g)=\varphi_{G}h(\Phi_{V})h(g)=\Psi_{G}\varphi_{G}h(g)=\Psi_{G}f\varphi_{F},
$$
i.e. we have $f\circ \Psi_{F}\circ \varphi_{F}=\Psi_{G}\circ f\circ \varphi_{F}$. Since $\varphi_{F}$ is surjective as a map of sheaves, this implies that $f\Psi_{F}=\Psi_{G}f$ as desired.
\end{proof}

Moreover, we show that every morphism of schemes induces a morphism of algebraic spaces such that the algebraic Frobenius induces a self-map of the identity functor on the category of algebraic spaces over $S$ and induced morphisms.
\begin{proposition}\label{A17}
Let $F,G$ be algebraic spaces of characteristic $p$ over $S$. Let $\varphi_{F}:h_{U}\rightarrow F$ and $\varphi_{G}:h_{V}\rightarrow G$ be surjective \'{e}tale maps where $U,V\in{\rm{Ob}}(({\rm Sch}/S)_{fppf})$ have characteristic $p$. Let $\Psi_{F}:F\rightarrow F$ and $\Psi_{G}:G\rightarrow G$ be the algebraic Frobenius morphisms of $F,G$. Suppose that $f:U\rightarrow V$ is a morphism of schemes over $S$. Then $f$ induces a morphism $f^{*}:F\rightarrow G$ of algebraic spaces over $S$ such that the diagram
$$
\xymatrix{
  F \ar[d]_{\Psi_{F}} \ar[r]^{f^{*}} & G \ar[d]^{\Psi_{G}} \\
  F \ar[r]^{f^{*}} & G   }
$$
commutes.
\end{proposition}
\begin{proof}
Consider the following solid diagram
$$
\xymatrix{
  h_{U} \ar[d]_{h(f)} \ar[r]^{\varphi_{F}} & F \ar@{-->}[d]^{f^{*}} \\
  h_{V} \ar[r]^{\varphi_{G}} & G   }
$$
By Lemma \ref{A10}, there exists a canonical map $f^{*}:F\rightarrow G$ such that the above diagram is completed into a commutative diagram, i.e. we have $\varphi_{G}\circ h(f)=f^{*}\circ \varphi_{F}$.

Now, consider the commutative diagrams
$$
\xymatrix{
  h_{U} \ar[d]_{h(\Phi_{U})} \ar[r]^{\varphi_{F}} & F \ar@{-->}[d]^{\Psi_{F}} \\
  h_{U} \ar[r]^{\varphi_{F}} & F   }
$$
and
$$
\xymatrix{
  h_{V} \ar[d]_{h(\Phi_{V})} \ar[r]^{\varphi_{G}} & G \ar@{-->}[d]^{\Psi_{G}} \\
  h_{V} \ar[r]^{\varphi_{G}} & G   }
$$
i.e. we have $\Psi_{F}\circ \varphi_{F}=\varphi_{F}\circ h(\Phi_{U})$ and $\Psi_{G}\circ \varphi_{G}=\varphi_{G}\circ h(\Phi_{V})$. Composing each side of the first equality by $f^{*}$ gives
$$
f^{*}\circ \Psi_{F}\circ \varphi_{F}=f^{*}\circ \varphi_{F}\circ h(\Phi_{U})=\varphi_{G}\circ h(f)\circ h(\Phi_{U})=\varphi_{G}\circ h(f\Phi_{U}).
$$
By Lemma \ref{L4}, we have $f\Phi_{U}=\Phi_{V}f$. Hence,
$$
\varphi_{G}h(f\Phi_{U})=\varphi_{G}h(\Phi_{V}f)=\varphi_{G}h(\Phi_{V})h(f)=\Psi_{G}\varphi_{G}h(f)=\Psi_{G}f^{*}\varphi_{F},
$$
i.e. we have $f^{*}\circ \Psi_{F}\circ \varphi_{F}=\Psi_{G}\circ f^{*}\circ \varphi_{F}$. Since $\varphi_{F}$ is surjective as a map of sheaves, this implies that $f^{*}\circ \Psi_{F}=\Psi_{G}\circ f^{*}$ as desired.
\end{proof}

We claim that an algebraic space $F$ is perfect if and only if one of its algebraic Frobenius morphisms $\Psi_{F}:F\rightarrow F$ is an isomorphism. To prove this assertion, we first observe the following lemma.
\begin{lemma}\label{A14}
Let $\calc$ be a site with enough points and $F,F',G,G'$ be sheaves of sets on $\calc$. Consider the following commutative diagram
$$
\xymatrix{
  G \ar[d]_{\varphi} \ar[r]^{f} & F \ar[d]^{h} \\
  G' \ar[r]^{g} & F'   }
$$
of morphisms of sheaves. Assume that $f,g$ are surjections of sheaves. Then $h:F\rightarrow F'$ is an isomorphism if and only if $\varphi$ is an isomorphism.
\end{lemma}
\begin{proof}
Let $\{p_{i}\}_{i\in I}$ be a conservative family of points and let $p_{i}\in\{p_{i}\}_{i\in I}$. Consider the commutative diagram of stalks
$$
\xymatrix{
  G_{p_{i}} \ar[d]_{\varphi_{p_{i}}} \ar[r]^{f_{p_{i}}} & F_{p_{i}} \ar[d]^{h_{p_{i}}} \\
  G'_{p_{i}} \ar[r]^{g_{p_{i}}} & F'_{p_{i}}   }
$$
By assumption, the maps $f_{p_{i}},g_{p_{i}}$ are surjective and the map $\varphi_{p_{i}}$ is bijective. The inverse $\varphi_{p_{i}}^{-1}$ of $\varphi_{p_{i}}$ induces the following commutative diagram
$$
\xymatrix{
  G_{p_{i}} \ar[d]_{\varphi_{p_{i}}} \ar[r]^{f_{p_{i}}} & F_{p_{i}} \ar[d]^{h_{p_{i}}} \\
  G'_{p_{i}} \ar[d]_{\varphi_{p_{i}}^{-1}} \ar[r]^{g_{p_{i}}} & F'_{p_{i}} \ar[d]^{h'_{p_{i}}} \\
  G_{p_{i}} \ar[r]^{f_{p_{i}}} & F_{p_{i}}   }
$$
Thus, we have $h'_{p_{i}}\circ h_{p_{i}}\circ f_{p_{i}}=f_{p_{i}}$. And since $f_{p_{i}}$ is surjective, this implies that $h'_{p_{i}}\circ h_{p_{i}}=id$. Similarly, we have $h_{p_{i}}\circ h'_{p_{i}}=id$. Hence, $h_{p_{i}}$ is bijective so that $h$ is an isomorphism of sheaves. The proof of the converse is similar.
\end{proof}

Applying the \lemref{A14} to the case of perfect algebraic spaces, we obtain the following desirable theorem.
\begin{theorem}\label{C5}
Let $F$ be an algebraic space over $S$ of characteristic $p$ and let $f:h_{U}\rightarrow F$ be a surjective \'{e}tale map where $U\in{\rm{Ob}}(({\rm Sch}/S)_{fppf})$ has characteristic $p$. Then the algebraic Frobenius morphism $\Psi_{F}:F\rightarrow F$ is an isomorphism if and only if the absolute Frobenius morphism $\Phi_{U}:U\rightarrow U$ is an isomorphism.
\end{theorem}
\begin{proof}
Note that the big fppf site $({\rm Sch}/S)_{fppf}$ has enough points. Since the functor $h$ is fully faithful, $\Phi_{U}:U\rightarrow U$ is an isomorphism if and only if $h(\Phi_{U})$ is an isomorphism. Moreover, $f$ is surjective as a map of sheaves. Thus, the theorem follows directly from Lemma \ref{A14}.
\end{proof}

The following corollary provides us with an alternative definition of perfect algebraic spaces.
\begin{corollary}
Let $F$ be an algebraic space over $S$ of characteristic $p$. Then $F$ is perfect if and only if one of the algebraic Frobenius morphisms $\Psi_{F}:F\rightarrow F$ of $F$ is an isomorphism.
\end{corollary}

\subsection{Relative algebraic Frobenius morphisms}\

Now, we construct the relative algebraic Frobenius morphism which is analogical to the relative absolute Frobenius morphism.

Let $F,G$ be algebraic spaces of characteristic $p$ over $S$, let $f:F\rightarrow G$ be a morphism of algebraic spaces over $S$, and let $\Psi_{F},\Psi_{G}$ be the algebraic Frobenius morphisms of $F,G$.
Consider the Cartesian diagram
$$
\xymatrix{
  F^{(p)} \ar[d]_{\textrm{pr}_{G}} \ar[r]^{\textrm{pro}_{F}} & F \ar[d]^{f^{*}} \\
  G \ar[r]^{\Psi_{G}} & G   }
$$
where $F^{(p)}:=F\times_{G,\Psi_{G}}G$. Then Proposition \ref{A17} yields another commutative diagram as follows.
$$
\xymatrix{
  F \ar[d]_{f} \ar[r]^{\Psi_{F}} & F \ar[d]^{f} \\
  G \ar[r]^{\Psi_{G}} & G   }
$$
By the universal property of the pullback, it yields a unique morphism $\Psi_{F/G}:F\rightarrow F^{(p)}$ of algebraic spaces such that the diagram
$$
\xymatrix{
  F \ar@/_/[ddr]_{f} \ar@/^/[drr]^{\Psi_{F}}
    \ar@{.>}[dr]|-{\Psi_{F/G}}                   \\
   & F^{(p)} \ar[d]^{\textrm{pr}_{G}} \ar[r]^{\textrm{pro}_{F}}
                      & F \ar[d]^{f}    \\
   & G \ar[r]^{\Psi_{G}}     & G               }
$$
commutes. Note that if $\Psi_{G}$ is an isomorphism, then we have $F^{(p)}\cong F$.

\begin{definition}
The unique morphism $\Psi_{F/G}:F\rightarrow F^{(p)}$ as above is called the relative algebraic Frobenius morphism of $F$ over $G$ with respect to $\Psi_{F}$.
\end{definition}

Observe that $F\mapsto F^{(p)}$ is a functor. The following lemma characterizes some properties of this functor.
\begin{lemma}
Let $F$ be an algebraic space over $S$ and let $U\in{\rm{Ob}}(({\rm Sch}/S)_{fppf})$. Then we have $h_{U^{(p)}}=h_{U}^{(p)}$. Moreover, if $h_{U}\rightarrow F$ is surjective \'{e}tale, then the induced map $h_{U}^{(p)}\rightarrow F$ is also surjective \'{e}tale.
\end{lemma}
\begin{proof}
Note that $h_{U^{(p)}}=h_{U\times_{S,\Phi_{S}}S}=h_{U}\times_{h_{S},\Psi_{h_{S}}}h_{S}=h_{U}^{(p)}$ where $\Phi_{S}$ is the absolute Frobenius morphism and $\Psi_{h_{S}}$ is the algebraic Frobenius morphism. By Yoneda lemma, the map $h_{U}^{(p)}\rightarrow h_{U}$ comes from a unique morphism of schemes $U\times_{S,\Phi_{S}}S\rightarrow U$. Since $S\rightarrow S$ is surjective, $U\times_{S,\Phi_{S}}S\rightarrow U$ is surjective. Thus, the composition $h_{U}^{(p)}\rightarrow h_{U}\rightarrow F$ is surjective.

Now, choose $h_{W}\simeq h_{V}\times_{F}h_{U}$ for some $W\in{\rm{Ob}}(({\rm Sch}/S)_{fppf})$. Since $h_{U}\rightarrow F$ is \'{e}tale, the morphism of schemes $W\rightarrow U$ is \'{e}tale. Then the base change $W\times_{U}(U\times_{S,\Phi_{S}}S)$ is \'{e}tale. Thus, the map $h_{U\times_{S,\Phi_{S}}S}=h_{U}^{(p)}\rightarrow h_{U}$ is \'{e}tale such that the composition $h_{U}^{(p)}\rightarrow F$ is \'{e}tale.
\end{proof}

The following proposition is analogous to Lemma \ref{L5}.
\begin{proposition}
Let $f:F\rightarrow F'$ be a morphism of algebraic spaces of characteristic $p$ over $S$, let $\Psi_{F},\Psi_{F'}$ be the algebraic Frobenius morphisms of $F,F'$, and let $\Psi_{F/G},\Psi_{F'/G}$ be the relative algebraic Frobenius morphisms of $F,F'$ with respect to $\Psi_{F},\Psi_{F'}$. Then there is a commutative diagram
$$
\xymatrix{
  F \ar[d]_{f} \ar[r]^{\Psi_{F/G}} & F^{(p)} \ar[d]^{f^{(p)}} \\
  F' \ar[r]^{\Psi_{F'/G}} & F'^{(p)}   }
$$
\end{proposition}
\begin{proof}
Proposition \ref{A6} gives the following commutative diagram.
$$
\xymatrix{
  F \ar[d]_{\Psi_{F}} \ar[r]^{f} & F' \ar[d]^{\Psi_{F'}} \\
  F \ar[r]^{f} & F'   }
$$
It follows that this diagram can be factored into the following one
$$
\xymatrix{
  F \ar[d]_{f} \ar[r]^{\Psi_{F/G}} & F^{(p)} \ar[d]^{f^{(p)}} \ar[r]^{\textrm{pro}_{F}} & F \ar[d]^{f} \\
  F' \ar[r]^{\Psi_{F'/G}} & F'^{(p)} \ar[r]^{\textrm{pro}_{F'}} & F'   }
$$
Thus, we get $f^{(p)}\circ \Psi_{F/G}=\Psi_{F'/G}\circ f$ as desired.
\end{proof}

Furthermore, the algebraic Frobenius morphisms commute with \'{e}tale localizations.
\begin{proposition}\label{P1}
Let $f:F\rightarrow G$ be a morphism of algebraic spaces in characteristic $p$ over $S$, let $\Psi_{F},\Psi_{G}$ be the algebraic Frobenius morphisms of $F,G$, and let $\Psi_{F/G}:F\rightarrow F^{(p)}$ be the relative algebraic Frobenius of $F$ with respect to $\Psi_{F}$. If $f$ is \'{e}tale, then $\Psi_{F/G}:F\rightarrow F^{(p)}$ is an isomorphism.
\end{proposition}
\begin{proof}
Choose a surjective \'{e}tale map $h_{U}\rightarrow F$ for $U\in{\rm{Ob}}(({\rm Sch}/S)_{fppf})$ of characteristic $p$. By Proposition \ref{A32}, $U\times_{F,\Psi_{F}}F$ and $U\times_{G,\Psi_{G}}G$ are schemes. It follows from \cite[Lemma A.2]{Zhu} that we have $U\rightarrow U\times_{F,\Psi_{F}}F\rightarrow U\times_{G,\Psi_{G}}G$, where the first map $U\rightarrow U\times_{F,\Psi_{F}}F$ and the composition are isomorphisms. Thus, the second map $U\times_{F,\Psi_{F}}F\rightarrow U\times_{G,\Psi_{G}}G$ is also an isomorphism. Since $U\times_{F,\Psi_{F}}F\rightarrow U\times_{G,\Psi_{G}}G$ is a base change of $F\rightarrow F\times_{G,\Psi_{G}}G$ via $U\times_{G,\Psi_{G}}G\rightarrow G$, we have that $F\rightarrow F\times_{G,\Psi_{G}}G$ is an isomorphism.
\end{proof}

\subsection{Properties of algebraic Frobenius morphisms}\

Next, we wonder if the algebraic Frobenius morphisms of an algebraic space have some desirable properties. We show that analogous to the absolute Frobenius morphisms of schemes, every algebraic Frobenius morphism is necessarily surjective.
\begin{proposition}\label{A30}
Let $F$ be an algebraic space of characteristic $p$ over $S$ with algebraic Frobenius morphism $\Psi_{F}:F\rightarrow F$. Then $\Psi_{F}$ is surjective.
\end{proposition}
\begin{proof}
Let $f:h_{U}\rightarrow F$ be a surjective \'{e}tale map for $U\in{\rm{Ob}}(({\rm Sch}/S)_{fppf})$ of characteristic $p$. Then there is a commutative diagram
$$
\xymatrix{
  h_{U} \ar[d]_{h(\Phi_{U})} \ar[r]^{f} & F \ar[d]^{\Psi_{F}} \\
  h_{U} \ar[r]^{f} & F   }
$$
Since $\Phi_{U}$ is surjective, $h(\Phi_{U})$ is surjective by \cite[Tag02WJ]{StackProject}. It follows from \cite[Tag03MF]{StackProject} that $\Psi_{F}$ is surjective.
\end{proof}
\begin{remark}
This does not necessarily imply that $\Psi_{F}$ is surjective as a map of sheaves or presheaves.
\end{remark}

The following proposition shows that algebraic Frobenius morphisms are representable.
\begin{proposition}\label{A32}
Let $F$ be an algebraic space of characteristic $p$ over $S$ and let $\Psi_{F}:F\rightarrow F$ be the algebraic Frobenius morphism of $F$. Then $\Psi_{F}$ is representable.
\end{proposition}
\begin{proof}
If $F$ is the empty algebraic space, then the algebraic Frobenius $\Psi_{F}:F\rightarrow F$ is trivially representable. Assume that $F\neq\varnothing$. Let $U,V\in {\rm{Ob}}(({\rm Sch}/S)_{fppf})$ be nonempty schemes and choose $h_{W}\simeq h_{U}\times_{F}h_{V}$ for some $W\in {\rm{Ob}}(({\rm Sch}/S)_{fppf})$ such that $h_{W}(X)\neq\varnothing$ for all $X\in {\rm{Ob}}(({\rm Sch}/S)_{fppf})$. Consider the composition $F\times_{F}h_{U}\rightarrow h_{U}\rightarrow h_{W}$. Then we have
$$(F\times_{F}h_{U})\times_{h_{W}}(h_{U}\times_{F}h_{V})\simeq F\times_{F}h_{U}.$$
Note that
$$(F\times_{F}h_{U})\times_{h_{W}}(h_{U}\times_{F}h_{V})\simeq(h_{U\times_{W}U}\times_{F}F)\times_{F}h_{V}\simeq h_{U\times_{W}U}\times_{F}h_{V}\simeq h_{L}$$
for some $L\in {\rm{Ob}}(({\rm Sch}/S)_{fppf})$. Thus, we get $F\times_{F}h_{U}\simeq h_{L}$ such that $\Psi_{F}$ is representable.
\end{proof}

Moreover, one can show that the algebraic Frobenius morphisms are affine. Therefore, it makes sense to form inverse limits of algebraic spaces whose transition maps are algebraic Frobenius morphisms.
\begin{proposition}
Let $F$ be an algebraic space of characteristic $p$ over $S$. Then the algebraic Frobenius morphism $\Psi_{F}:F\rightarrow F$ is affine.
\end{proposition}
\begin{proof}
Let $Z$ be an affine scheme over $S$ with \'{e}tale morphism $Z\rightarrow F$. Let $U\rightarrow F$ be a surjective \'{e}tale map where $U\in {\rm{Ob}}(({\rm Sch}/S)_{fppf})$ has characteristic $p$. By \cite[Tag02X1]{StackProject}, there exist $V\in {\rm{Ob}}(({\rm Sch}/S)_{fppf})$ and a commutative diagram
$$
\xymatrix{
  V \ar[d]_{} \ar[r]^{} & U \ar[d]^{} \\
  Z \ar[r]^{} & F   }
$$
where $V\rightarrow Z$ is surjective \'{e}tale. This shows that $Z$ has weak characteristic $p$. Thus, it follows from Proposition \ref{P1} that we have an isomorphism $F\times_{\Psi_{F},F}Z\cong Z$, which shows that $\Psi_{F}$ is affine.
\end{proof}

We feel that the above results have a deeper meaning. So we first study certain kinds of morphisms of algebraic spaces in the following section.
\section{Perfect morphisms}\label{B4}
In this section, we study certain kinds of morphisms of functors that are discovered in the previous section.
\begin{definition}
Let $F,G:({\rm Sch}/S)_{fppf}^{opp}\rightarrow {\bf Sets}$ be functors and let $f:F\rightarrow G$ be a morphism of functors over $S$. Then
\begin{itemize}
\item[(1)]
$f$ is said to be perfect if for every $U\in{\rm{Ob}}(({\rm Sch}/S)_{fppf})$ and any $\xi\in G(U)$, the functor $h_{U}\times_{G}F$ is represented by a perfect scheme $W\in{\rm{Ob}}(({\rm Sch}/S)_{fppf})$.
\item[(2)]
$f$ is said to be quasi-perfect if there exists perfect scheme $U\in{\rm{Ob}}(({\rm Sch}/S)_{fppf})$ such that for any $\xi\in G(U)$, the functor $h_{U}\times_{G}F$ is represented by a perfect scheme $W\in{\rm{Ob}}(({\rm Sch}/S)_{fppf})$.
\item[(3)]
$f$ is said to be semiperfect if there exists $U\in{\rm{Ob}}(({\rm Sch}/S)_{fppf})$ such that for any $\xi\in G(U)$, the functor $h_{U}\times_{G}F$ is represented by a perfect scheme $W\in{\rm{Ob}}(({\rm Sch}/S)_{fppf})$.
\item[(4)]
$f$ is said to be weakly perfect if for every perfect scheme $U\in{\rm{Ob}}(({\rm Sch}/S)_{fppf})$ and any $\xi\in G(U)$, the functor $h_{U}\times_{G}F$ is represented by a perfect scheme $W\in{\rm{Ob}}(({\rm Sch}/S)_{fppf})$.
\end{itemize}
If $f$ is perfect (resp. quasi-perfect, semiperfect, weakly perfect), then we say that $F$ is perfect (resp. quasi-perfect, semiperfect, weakly perfect) over $G$.
\end{definition}
Obviously, every perfect morphism of functors is quasi-perfect, semiperfect, and weakly perfect. Meanwhile, every quasi-perfect morphism of functors is semiperfect. And every weakly perfect morphism is both quasi-perfect and semiperfect.

Perfect (resp. quasi-perfect, semiperfect, weakly perfect) morphisms of functors are stable under compositions.
\begin{proposition}\label{A24}
Let $F,G,H:({\rm Sch}/S)_{fppf}^{opp}\rightarrow {\bf{Sets}}$ be functors and let $f:F\rightarrow G$ and $g:G\rightarrow H$ be morphisms of functors over $S$.
\begin{itemize}
\item[(1)]
If $f,g$ are perfect, then the composition $g\circ f:F\rightarrow H$ is also perfect.
\item[(2)]
If $f,g$ are quasi-perfect, then the composition $g\circ f:F\rightarrow H$ is also quasi-perfect.
\item[(3)]
If $f,g$ are semiperfect, then the composition $g\circ f:F\rightarrow H$ is also semiperfect.
\item[(4)]
If $f,g$ are weakly perfect, then the composition $g\circ f:F\rightarrow H$ is also weakly perfect.
\end{itemize}
\end{proposition}
\begin{proof}
(1) Let $U\in{\rm{Ob}}(({\rm Sch}/S)_{fppf})$ and $\xi\in G(U),\xi'\in H(U)$. Choose $h_{W}\simeq h_{U}\times_{\xi,G}F$ and $h_{W'}\simeq h_{U}\times_{\xi',H}G$ for some perfect schemes $W,W'\in{\rm{Ob}}(({\rm Sch}/S)_{fppf})$. Then we have
$$(h_{U}\times_{G}F)\times_{h_{U}}(h_{U}\times_{H}G)\simeq h_{W}\times_{h_{U}}h_{W'}=h_{W\times_{U}W'},$$
where $W\times_{U}W'$ is perfect due to Proposition \ref{A1}. Note that
$$(h_{U}\times_{G}F)\times_{h_{U}}(h_{U}\times_{H}G)\simeq h_{U}\times_{G}F\times_{H}G\simeq h_{U}\times_{H}F.$$
Hence, we get $h_{U}\times_{H}F\simeq h_{W\times_{U}W'}$ as desired. This implies that $g\circ f$ is perfect.

(2) Let $U,V\in{\rm{Ob}}(({\rm Sch}/S)_{fppf})$ be perfect schemes and $\xi\in G(U),\xi'\in H(U)$. Choose isomorphisms $h_{W}\simeq h_{U}\times_{\xi,G}F$ and $h_{W'}\simeq h_{V}\times_{\xi',H}G$ for some perfect schemes $W,W'\in{\rm{Ob}}(({\rm Sch}/S)_{fppf})$. Since $V\neq\varnothing$, $h_{V}(X)\neq\varnothing$ for all $X\in {\rm{Ob}}(({\rm Sch}/S)_{fppf})$, it follows that the composition $h_{U}\times_{G}F\rightarrow h_{U}\rightarrow h_{V}$ makes sense. Then we have
$$(h_{U}\times_{G}F)\times_{h_{V}}(h_{V}\times_{H}G)\simeq h_{W}\times_{h_{V}}h_{W'}=h_{W\times_{V}W'}.$$
Further, Proposition \ref{A1} shows that $W\times_{V}W'$ is perfect, and there are natural isomorphisms
$$(h_{U}\times_{G}F)\times_{h_{V}}(h_{V}\times_{H}G)\simeq h_{U}\times_{G}F\times_{H}G\simeq h_{U}\times_{H}F.$$ Thus, we obtain $h_{U}\times_{H}F\simeq h_{W\times_{V}W'}$ such that $g\circ f$ is quasi-perfect.

(3) Let $U,V\in{\rm{Ob}}(({\rm Sch}/S)_{fppf})$ and let $\xi\in G(U),\xi'\in H(U)$. Choose isomorphisms $h_{W}\simeq h_{U}\times_{\xi,G}F$ and $h_{W'}\simeq h_{V}\times_{\xi',H}G$ for some perfect schemes $W,W'\in{\rm{Ob}}(({\rm Sch}/S)_{fppf})$. Clearly, $U,V\neq\varnothing$. This shows that one can make the composition $h_{U}\times_{G}F\rightarrow h_{U}\rightarrow h_{V}$. Then  $$(h_{U}\times_{G}F)\times_{h_{V}}(h_{V}\times_{H}G)\simeq h_{W}\times_{h_{V}}h_{W'}.$$
 Since $({\rm Sch}/S)_{fppf}$ admits fibre products, $h_{W}\times_{h_{V}}h_{W'}=h_{W\times_{V}W'}$. Further, by Proposition \ref{A1}, $W\times_{V}W'$ is perfect, and there are natural isomorphisms $$(h_{U}\times_{G}F)\times_{h_{V}}(h_{V}\times_{H}G)\simeq h_{U}\times_{G}F\times_{H}G\simeq h_{U}\times_{H}F.$$
 Thus, we obtain $h_{U}\times_{H}F\simeq h_{W\times_{V}W'}$ such that $g\circ f$ is semiperfect.

(4) This is similar to the proof of (1).
\end{proof}

Perfect (resp. quasi-perfect, semiperfect, weakly perfect) morphisms of functors are stable under arbitrary base change.
\begin{proposition}\label{P4}
Let $F,G,H:({\rm Sch}/S)_{fppf}^{opp}\rightarrow {\bf Sets}$ be functors, and gives the following fibre product diagram
$$
\xymatrix{
  F\times_{H}G \ar[d]_{a'} \ar[r]^{\ \ b'} & F \ar[d]^{a} \\
  G \ar[r]^{b} & H   }
$$
If $a$ is perfect (resp. quasi-perfect, semiperfect, weakly perfect), then the base change $a'$ is also perfect (resp. quasi-perfect, semiperfect, weakly perfect).
\end{proposition}
\begin{proof}
Let $U\in{\rm{Ob}}(({\rm Sch}/S)_{fppf})$ and let $\xi\in H(U)$. Choose $h_{W}\simeq h_{U}\times_{H}F$ for some perfect scheme $W\in{\rm{Ob}}(({\rm Sch}/S)_{fppf})$. Then for any $\xi'\in G(U)$, the natural isomorphisms
$$h_{U}\times_{G}(F\times_{H}G)\simeq h_{U}\times_{H}F\simeq h_{W}$$
shows that $a'$ is perfect.

If $U\in{\rm{Ob}}(({\rm Sch}/S)_{fppf})$ is a perfect scheme, then for any $\xi\in G(U)$, there are also natural isomorphisms
$$h_{U}\times_{G}(F\times_{H}G)\simeq h_{U}\times_{H}F\simeq h_{W},$$
which shows that $a'$ is quasi-perfect. Similarly, we can show that if $a$ is semiperfect (resp. weakly perfect), then $a'$ is also semiperfect (resp. weakly perfect).
\end{proof}

As a consequence, we have the following proposition.
\begin{proposition}\label{P5}
Let $F_{i},G_{i}:({\rm Sch}/S)_{fppf}^{opp}\rightarrow {\bf Sets}$ be functors and $a_{i}:F_{i}\rightarrow G_{i}$ be morphisms of functors over $S$ for $i=1,2$. If each $a_{i}$ is perfect (resp. quasi-perfect, semiperfect, weakly perfect), then
$$
a_{1}\times a_{2}:F_{1}\times F_{2}\rightarrow G_{1}\times G_{2}
$$
is perfect (resp. quasi-perfect, semiperfect, weakly perfect).
\end{proposition}
\begin{proof}
The morphism $a_{1}\times a_{2}$ can be expressed as the composition $F_{1}\times F_{2}\rightarrow G_{1}\times F_{2}\rightarrow G_{1}\times G_{2}$. Note that $F_{1}\times F_{2}=F_{1}\times_{G_{1}}(G_{1}\times F_{2})$ and $G_{1}\times F_{2}=(G_{1}\times G_{2})\times_{G_{2}}F_{2}$. Thus, $F_{1}\times F_{2}\rightarrow G_{1}\times F_{2}$ is the base change of $a_{1}$ via $G_{1}\times F_{2}\rightarrow G_{1}$ and $G_{1}\times F_{2}\rightarrow G_{1}\times G_{2}$ is the base change of $a_{2}$ via $G_{1}\times G_{2}\rightarrow G_{2}$. Then the results follow by Proposition \ref{A24} and \ref{P4}.
\end{proof}

We give the following theorem which specifies that there are equivalences between perfect algebraic spaces and perfect morphisms of algebraic spaces.
\begin{theorem}\label{C2}
Let $F$ be an algebraic space over $S$. Then
\begin{itemize}
\item[(1)]
$F$ is quasi-perfect if and only if there exists a perfect scheme $U\in{\rm{Ob}}(({\rm Sch}/S)_{fppf})$ such that every map $\varphi_{F}:h_{U}\rightarrow F$ is a quasi-perfect morphism of algebraic spaces over $S$. Equivalently, the diagonal $\Delta:F\rightarrow F\times F$ is quasi-perfect.
\item[(2)]
$F$ is semiperfect if and only if there exists $U\in{\rm{Ob}}(({\rm Sch}/S)_{fppf})$ such that every map $\varphi_{F}:h_{U}\rightarrow F$ is a semiperfect morphism of algebraic spaces over $S$. Equivalently, the diagonal $\Delta:F\rightarrow F\times F$ is semiperfect.
\item[(3)]
$F$ is strongly perfect if and only if for every perfect scheme $U\in{\rm{Ob}}(({\rm Sch}/S)_{fppf})$, the maps $\varphi_{F}:h_{U}\rightarrow F$ are weakly perfect. Equivalently, the diagonal $\Delta:F\rightarrow F\times F$ is weakly perfect.
\end{itemize}
\end{theorem}
\begin{proof}
(1) Let $U,V\in{\rm{Ob}}(({\rm Sch}/S)_{fppf})$ be perfect schemes and $\xi\in F(V),\xi'\in F(U)$ such that the functor $h_{V}\times_{F}h_{U}$ is represented by a perfect scheme. In other words, there exists perfect scheme $U\in{\rm{Ob}}(({\rm Sch}/S)_{fppf})$ such that the map $\varphi_{F}:h_{U}\rightarrow F$ is quasi-perfect. Conversely, there exists a perfect scheme $V\in{\rm{Ob}}(({\rm Sch}/S)_{fppf})$ such that for any $\xi\in F(V)$ the functor $h_{V}\times_{F}h_{U}$ is represented by a perfect scheme. In other words, there exist perfect schemes $U,V\in{\rm{Ob}}(({\rm Sch}/S)_{fppf})$ and $\xi\in F(V),\xi'\in F(U)$ such that the functor $h_{V}\times_{F}h_{U}$ is represented a perfect scheme. Hence, $F$ is quasi-perfect.

Equivalently, suppose that $F\rightarrow F\times F$ is quasi-perfect. Let $U,V\in{\rm{Ob}}(({\rm Sch}/S)_{fppf})$ be perfect schemes and let $c:h_{U\times V}\rightarrow F\times F$ be the map $a\times b:h_{U}\times h_{V}\rightarrow F\times F$ for any $a\in F(U),b\in F(V)$. Then the fibre product $h_{U\times V}\times_{c,F\times F,\Delta}F\simeq h_{W}$ for a perfect scheme $W\in {\rm{Ob}}(({\rm Sch}/S)_{fppf})$. Hence there exist morphisms $h_{W}\rightarrow h_{U\times V}$ and $h_{W}\rightarrow F$ such that the diagram
$$
\xymatrix{
  h_{W} \ar[d]_{} \ar[r]^{} & h_{U\times V}  \ar[d]^{a\times b} \\
  F \ar[r]^{} & F\times F   }
$$
is Cartesian. By \cite[Tag0022]{StackProject}, we have $h_{W}=h_{U}\times_{a,F,b}h_{V}$ so that $F$ is quasi-perfect.

On the other hand, assume that every map $\varphi_{F}:h_{U}\rightarrow F$ is a quasi-perfect morphism for a perfect scheme $U\in{\rm{Ob}}(({\rm Sch}/S)_{fppf})$. Then for some perfect scheme $V\in{\rm{Ob}}(({\rm Sch}/S)_{fppf})$ and any $a\in F(V)$, there exists a perfect scheme $W\in{\rm{Ob}}(({\rm Sch}/S)_{fppf})$ such that $h_{U}\times_{\varphi_{F},F,a}h_{V}\simeq h_{W}$. By Yoneda lemma, the compositions $$h_{W}\rightarrow h_{U}\times_{\varphi_{F},F,a}h_{V}\rightarrow h_{U}\ {\rm and}\ h_{W}\rightarrow h_{U}\times_{\varphi_{F},F,a}h_{V}\rightarrow h_{V}$$
come from the canonical maps $p_{1}:W\rightarrow U$ and $p_{2}:W\rightarrow V$. Consider the fibre product $W'=W\times_{(p_{1},p_{2}),U\times V}U$. By Proposition \ref{A1}, $W'$ is a perfect scheme. It follows that
$$
h_{W'}=h_{W}\times_{h_{U}\times h_{V}}h_{U}=(h_{U}\times_{\varphi_{F},F,a}h_{V})\times_{h_{U}\times h_{V}}h_{U}=F\times_{F\times F}h_{U}.
$$
Thus, $F\rightarrow F\times F$ is quasi-perfect.

(2) Let $U,V\in{\rm{Ob}}(({\rm Sch}/S)_{fppf})$ and $\xi\in F(V),\xi'\in F(U)$ such that the functor $h_{V}\times_{F}h_{U}$ is represented by a perfect scheme. In other words, there exists $U\in{\rm{Ob}}(({\rm Sch}/S)_{fppf})$ such that every map $\varphi_{F}:h_{U}\rightarrow F$ is represented by a perfect scheme and is semiperfect.

Conversely, there exists $V\in{\rm{Ob}}(({\rm Sch}/S)_{fppf})$ such that for any $\xi\in F(V)$ the functor $h_{V}\times_{F}h_{U}$ is represented by a perfect scheme. In other words, there exist $U,V\in{\rm{Ob}}(({\rm Sch}/S)_{fppf})$ and $\xi\in F(V),\xi'\in F(U)$ such that the functor $h_{V}\times_{F}h_{U}$ is represented by a perfect scheme. Hence, $F$ is semiperfect.

The proof of the equivalent statement is similar to (1).

(3) If $F$ is strongly perfect, then for every perfect scheme $V\in{\rm{Ob}}(({\rm Sch}/S)_{fppf})$, the maps $h_{V}\rightarrow F$ are weakly perfect. Conversely, if for every perfect schemes $U,V\in{\rm{Ob}}(({\rm Sch}/S)_{fppf})$ and any $\xi\in F(V),\xi'\in F(U)$, the functor $h_{V}\times_{F}h_{U}$ is represented by a perfect scheme. Then $F$ is strongly perfect.

The proof of the equivalent statement is similar to (1).
\end{proof}

Now, we can complete our proof of \proref{P6} that the category $\textrm{SPerf}_{AP_{S}}$ (resp. $\textrm{QPerf}_{AP_{S}}$, $\textrm{StPerf}_{AP_{S}}$) of semiperfect (resp. quasi-perfect, strongly perfect) algebraic spaces over $S$ is stable under fibre products.
\begin{proof}[Proof of Proposition \ref{P6}]
(1) Since $F,G$ are semiperfect, it follows from Theorem \ref{C2} that $\Delta_{F}:F\rightarrow F\times F$ and $\Delta_{G}:G\rightarrow G\times G$ are semiperfect. Then we have the following diagram
$$
\xymatrix{
  F\times_{H}G \ar[d]_{\Delta} \ar[r]^{} & F\times G \ar[d]^{\Delta_{F}\times\Delta_{G}} \\
  (F\times F)\times_{H\times H}(G\times G) \ar[r]^{} & (F\times F)\times(G\times G)   }
$$
which is Cartesian. Note that $\Delta$ is the diagonal of $F\times_{H}G$. By Proposition \ref{P5}, $\Delta_{F}\times\Delta_{G}$ is semiperfect. Then
$$\Delta:F\times_{H}G\rightarrow(F\times F)\times_{H\times H}(G\times G)$$
is also semiperfect following Proposition \ref{P4}. Thus, $F\times_{H}G$ is semiperfect by Theorem \ref{C2}.

To prove (2) and (3), one just need to replace ``semiperfect'' by ``quasi-perfect'' or ``strongly perfect''.
\end{proof}

However, there is a kind of algebraic spaces that enjoys an equivalence to perfect morphisms. But such kind of algebraic spaces does not generalize perfect schemes. In fact, such kind of algebraic spaces does not even exist when they are representable.
\begin{definition}
Let $F$ be an algebraic space over $S$. Then $F$ is said to be pseudo-perfect if for every $U,V\in{\rm{Ob}}(({\rm Sch}/S)_{fppf})$ and any $\xi\in F(U),\xi'\in F(V)$, the functor $h_{U}\times_{\xi,F,\xi'}h_{V}$ is represented by a perfect scheme.
\end{definition}

It is obvious that every pseudo-perfect algebraic space is semiperfect, quasi-perfect, and strongly perfect. Let $\textrm{P-Perf}_{AP_{S}}$ be the category of pseudo-perfect algebraic spaces over $S$. Then there are full embeddings
\begin{align}
&\textrm{P-Perf}_{AP_{S}}\subset AP_{S}, \\
&\textrm{P-Perf}_{AP_{S}}\subset\textrm{StPerf}_{AP_{S}}\subset\textrm{SPerf}_{AP_{S}}\subset \textrm{QPerf}_{AP_{S}}.
\end{align}

\begin{proposition}
Let $F$ be an algebraic space over $S$. Then $F$ is pseudo-perfect if and only if for every $U\in{\rm{Ob}}(({\rm Sch}/S)_{fppf})$, the maps $h_{U}\rightarrow F$ are perfect.
\end{proposition}
\begin{proof}
Assume that $F$ is pseudo-perfect. Let $U,V\in{\rm{Ob}}(({\rm Sch}/S)_{fppf})$ and $\xi\in F(U),\xi'\in F(V)$ such that the functor $h_{V}\times_{F}h_{U}$ is represented by a perfect scheme. This induces a representable map $h_{U}\rightarrow F$ which is obviously perfect.

Conversely, let $h_{U}\rightarrow F$ be a perfect map. Then for every $V\in{\rm{Ob}}(({\rm Sch}/S)_{fppf})$ and $\xi'\in F(V)$, the functor $h_{V}\times_{F}h_{U}$ is represented by a perfect scheme. Thus, if any map $h_{U}\rightarrow F$ is perfect for every $U\in{\rm{Ob}}(({\rm Sch}/S)_{fppf})$, then $F$ is pseudo-perfect.
\end{proof}

Let us turn back to \S\ref{B3}. We will show that the algebraic Frobenius morphism of a perfect algebraic space is weakly perfect. We first observe that every isomorphism of algebraic spaces is weakly perfect.
\begin{lemma}\label{LL1}
Any isomorphism $f:F\rightarrow G$ of algebraic spaces over $S$ is weakly perfect. In particular, for any algebraic space $F$ over $S$, the identity morphism $1_{F}:F\rightarrow F$ is weakly perfect.
\end{lemma}
\begin{proof}
Let $U\in{\rm{Ob}}(({\rm Sch}/S)_{fppf})$. Then by Lemma \ref{A5}, we have $h_{U}\times_{G}F=h_{U}\times_{F}F\simeq h_{U}$. Since we may choose $U$ to be a perfect scheme, $f$ is weakly perfect.
\end{proof}

This gives rise to the following lemma.
\begin{lemma}\label{A28}
Let $U\in{\rm{Ob}}(({\rm Sch}/S)_{fppf})$ be a scheme of characteristic $p$ and let $\Phi_{U}$ be its absolute Frobenius morphism. If $U$ is perfect, then $h(\Phi_{U})$ is weakly perfect.
\end{lemma}
\begin{proof}
If $U$ is perfect, then $h(\Phi_{U})$ is an isomorphism. Thus, it follows from Lemma \ref{LL1} that $h(\Phi_{U})$ is weakly perfect.
\end{proof}

The following proposition shows that the algebraic Frobenius of a perfect algebraic space is weakly perfect.
\begin{proposition}\label{P3}
Let $F$ be an algebraic space in characteristic $p$ over $S$ and $\Psi_{F}:F\rightarrow F$ be the algebraic Frobenius morphism of $F$. If $F$ is perfect, then $\Psi_{F}$ is weakly perfect.
\end{proposition}
\begin{proof}
Let $\varphi_{F}:h_{U}\rightarrow F$ be a surjective \'{e}tale map. Here is the commutative diagram.
$$
\xymatrix{
  h_{U} \ar[d]_{h(\Phi_{U})} \ar[r]^{\varphi_{F}} & F \ar[d]^{\Psi_{F}} \\
  h_{U} \ar[r]^{\varphi_{F}} & F   }
$$
Since $F$ is perfect, $\Psi_{F}$ is an isomorphism by Theorem \ref{C5}. Thus, it follows from Lemma \ref{LL1} that $\Psi_{F}$ is weakly perfect.
\end{proof}

Let $k$ be a perfect field of characteristic $p$. If an algebraic space is \'{e}tale over $k$, then it is weakly perfect over $k$.
\begin{proposition}
Let $X$ be an algebraic space over $S$ with a morphism $\varphi:X\rightarrow{\rm{Spec}}(k)$. If $\varphi$ is \'{e}tale, then $\varphi$ is weakly perfect.
\end{proposition}
\begin{proof}
Let $V,W\in {\rm{Ob}}(({\rm Sch}/S)_{fppf})$ such that there is an \'{e}tale morphism $V\rightarrow\textrm{Spec}(k)$. Then the morphism $W\simeq V\times_{\textrm{Spec}(k)}X\rightarrow V$ is \'{e}tale by assumption. By the composition $W \rightarrow V\rightarrow\textrm{Spec}(k)$ of \'{e}tale morphisms, we obtain an \'{e}tale morphism $W\rightarrow\textrm{Spec}(k)$. Thus, $W,V$ is perfect by Proposition \ref{A13} so that $\varphi$ is weakly perfect.
\end{proof}

The following proposition shows that certain presheaves semiperfect over an algebraic space would be perfect.
\begin{proposition}
Let $F$ be an algebraic space over $S$ and let $f:G\rightarrow F$ be a representable semiperfect natural transformation of functors. Suppose that there exists a surjective \'{e}tale map $h_{U}\rightarrow F$ such that $h_{U}\times_{F}G$ is represented by a perfect scheme. Then $G$ is a perfect algebraic space.
\end{proposition}
\begin{proof}
It follows from \cite[Tag02WY]{StackProject} that $G$ is an algebraic space. By assumption, the base change $h_{U}\times_{F}G\rightarrow G$ is surjective \'{e}tale and there is a perfect scheme $W\in {\rm{Ob}}(({\rm Sch}/S)_{fppf})$ such that $h_{W}\simeq h_{U}\times_{F}G\rightarrow G$. Thus, $G$ is a perfect algebraic space.
\end{proof}

\section{Some categories of perfect algebraic spaces}\label{B5}
In this section, we construct several subcategories or subsemicategories of $AP_{S}$ spanned by perfect (resp. quasi-perfect, semiperfect, strongly perfect) algebraic spaces and perfect (resp. quasi-perfect, semiperfect, weakly perfect) morphisms, rather than the full subcategories of $AP_{S}$ in \S\ref{B2}.

\begin{definition1}\label{A31}
Here is a list of subcategories of the category $AP_{S}$ of algebraic spaces over $S$.
\begin{enumerate}
\item
The category $\mathcal{Q}\textrm{Perf}_{S}$ whose objects are quasi-perfect algebraic spaces over $S$ and morphisms are quasi-perfect morphisms in $AP_{S}$.
\item
The category $\mathcal{Q}\textrm{Perf}_{S}^{*}$ whose objects are quasi-perfect algebraic spaces over $S$ and morphisms are semiperfect morphisms in $AP_{S}$.
\item
The category $\mathcal{Q}\textrm{Perf}_{S}^{\sharp}$ whose objects are quasi-perfect algebraic spaces over $S$ and morphisms are weakly perfect morphisms in $AP_{S}$.
\item
The category $\mathcal{S}\textrm{Perf}_{S}$ whose objects are semiperfect algebraic spaces over $S$ and morphisms are semiperfect morphisms in $AP_{S}$.
\item
The category $\mathcal{S}\textrm{Perf}_{S}^{\sharp}$ whose objects are semiperfect algebraic spaces over $S$ and morphisms are quasi-perfect morphisms in $AP_{S}$.
\item
The category $\mathcal{S}\textrm{Perf}_{S}^{\sharp\sharp}$ whose objects are semiperfect algebraic spaces over $S$ and morphisms are weakly perfect morphisms in $AP_{S}$.
\item
The category $\mathcal{ST}\textrm{Perf}_{S}$ whose objects are strongly perfect algebraic spaces over $S$ and morphisms are weakly perfect morphisms in $AP_{S}$.
\item
The category $\mathcal{ST}\textrm{Perf}_{S}^{*}$ whose objects are strongly perfect algebraic spaces over $S$ and morphisms are quasi-perfect morphisms in $AP_{S}$.
\item
The category $\mathcal{ST}\textrm{Perf}_{S}^{**}$ whose objects are strongly perfect algebraic spaces over $S$ and morphisms are semiperfect morphisms in $AP_{S}$.
\item
The category $\textrm{Perf}^{\mathcal{Q}}_{S}$ whose objects are perfect algebraic spaces over $S$ and morphisms are quasi-perfect morphisms in $AP_{S}$.
\item
The category $\textrm{Perf}^{\mathcal{S}}_{S}$ whose objects are perfect algebraic spaces over $S$ and morphisms are semiperfect morphisms in $AP_{S}$.
\item
The category $\textrm{Perf}^{\mathcal{ST}}_{S}$ whose objects are perfect algebraic spaces over $S$ and morphisms are weakly perfect morphisms in $AP_{S}$.
\end{enumerate}
\end{definition1}

Then there are inclusion functors
\begin{align}
\mathcal{Q}\textrm{Perf}_{S}^{*}&\longrightarrow\textrm{QPerf}_{AP_{S}} \\
\mathcal{S}\textrm{Perf}_{S}&\longrightarrow\textrm{SPerf}_{AP_{S}} \\
\mathcal{ST}\textrm{Perf}_{S}^{**}&\longrightarrow\textrm{StPerf}_{AP_{S}} \\
\textrm{Perf}^{\mathcal{Q}}_{S}&\longrightarrow\textrm{Perf}_{AP_{S}}\label{I3} \\
\textrm{Perf}^{\mathcal{S}}_{S}&\longrightarrow\textrm{Perf}_{AP_{S}} \label{I4}\\
\textrm{Perf}^{\mathcal{ST}}_{S}&\longrightarrow\textrm{Perf}_{AP_{S}}\label{I5}
\end{align}
together with strings of inclusion functors following \S\ref{B2}.
\begin{align}
&\mathcal{ST}\textrm{Perf}_{S}\longrightarrow \mathcal{Q}\textrm{Perf}_{S}\longrightarrow\mathcal{S}\textrm{Perf}_{S} \\
&\mathcal{Q}\textrm{Perf}_{S}^{\sharp}\longrightarrow\mathcal{Q}\textrm{Perf}_{S}\longrightarrow\mathcal{Q}\textrm{Perf}_{S}^{*} \\
&\mathcal{S}\textrm{Perf}_{S}^{\sharp\sharp}\longrightarrow\mathcal{S}\textrm{Perf}_{S}^{\sharp}\longrightarrow\mathcal{S}\textrm{Perf}_{S} \\
&\mathcal{ST}\textrm{Perf}_{S}\longrightarrow\mathcal{ST}\textrm{Perf}_{S}^{*}\longrightarrow\mathcal{ST}\textrm{Perf}_{S}^{**} \\
&\textrm{Perf}^{\mathcal{ST}}_{S}\longrightarrow\textrm{Perf}^{\mathcal{S}}_{S}\longrightarrow\textrm{Perf}^{\mathcal{Q}}_{S}
\end{align}

Furthermore, we have the following commutative diagram of inclusion functors
$$
\xymatrix{
  \textrm{StPerf}_{AP_{S}}  \ar[r]^{} & \textrm{QPerf}_{AP_{S}}  \ar[r]^{} & \textrm{SPerf}_{AP_{S}}  \\
  \mathcal{ST}\textrm{Perf}_{S}^{**} \ar[u]_{}  \ar[r]^{} & \mathcal{Q}\textrm{Perf}_{S}^{*} \ar[u]_{}  \ar[r]^{} & \mathcal{S}\textrm{Perf}_{S} \ar[u]_{}  \\
  \mathcal{ST}\textrm{Perf}_{S}^{*} \ar[u]_{} \ar[r]^{} & \mathcal{Q}\textrm{Perf}_{S} \ar[u]_{}  \ar[r]^{} & \mathcal{S}\textrm{Perf}_{S}^{\sharp} \ar[u]_{}  \\
  \mathcal{ST}\textrm{Perf}_{S} \ar[u]_{} \ar[r]^{} & \mathcal{Q}\textrm{Perf}_{S}^{\sharp} \ar[u]_{} \ar[r]^{} & \mathcal{S}\textrm{Perf}_{S}^{\sharp\sharp} \ar[u]_{}   }
$$

Next, we can consider subcategories spanned by representable algebraic spaces.
\begin{definition1}\label{A34}
Here is a list of subcategories of the categories in Definition \ref{A31}:
\begin{enumerate}
\item
The category $\widehat{\mathcal{Q}\textrm{Perf}_{S}}$ of representable quasi-perfect algebraic spaces and quasi-perfect morphisms.
\item
The category $\widehat{\mathcal{Q}\textrm{Perf}_{S}^{*}}$ of representable quasi-perfect algebraic spaces and semiperfect morphisms.
\item
The category $\widehat{\mathcal{Q}\textrm{Perf}_{S}^{\sharp}}$ of representable quasi-perfect algebraic spaces and weakly perfect morphisms.
\item
The category $\widehat{\mathcal{S}\textrm{Perf}_{S}}$ of representable semiperfect algebraic spaces and semiperfect morphisms.
\item
The category $\widehat{\mathcal{S}\textrm{Perf}_{S}^{\sharp}}$ of representable semiperfect algebraic spaces and quasi-perfect morphisms.
\item
The category $\widehat{\mathcal{S}\textrm{Perf}_{S}^{\sharp\sharp}}$ of representable semiperfect algebraic spaces and weakly perfect morphisms.
\item
The category $\widehat{\mathcal{ST}\textrm{Perf}_{S}}$ whose objects are representable strongly perfect algebraic spaces over $S$ and morphisms are weakly perfect morphisms in $AP_{S}$.
\item
The category $\widehat{\mathcal{ST}\textrm{Perf}_{S}^{*}}$ whose objects are representable strongly perfect algebraic spaces over $S$ and morphisms are quasi-perfect morphisms in $AP_{S}$.
\item
The category $\widehat{\mathcal{ST}\textrm{Perf}_{S}^{**}}$ whose objects are representable strongly perfect algebraic spaces over $S$ and morphisms are semiperfect morphisms in $AP_{S}$.
\item
The category $\widehat{\textrm{Perf}^{\mathcal{Q}}_{S}}$ whose objects are representable perfect algebraic spaces over $S$ and morphisms are quasi-perfect morphisms in $AP_{S}$.
\item
The category $\widehat{\textrm{Perf}^{\mathcal{S}}_{S}}$ whose objects are representable perfect algebraic spaces over $S$ and morphisms are semiperfect morphisms in $AP_{S}$.
\item
The category $\widehat{\textrm{Perf}^{\mathcal{ST}}_{S}}$ whose objects are representable perfect algebraic spaces over $S$ and morphisms are weakly perfect morphisms in $AP_{S}$.
\end{enumerate}
\end{definition1}

Although we can not form a category whose morphisms are perfect morphisms of algebraic spaces, there is another suitable structure called the semicategory. It is a category without identity morphisms.
\begin{definition}[\cite{Garraway}, \cite{Hayashi}]
A semicategory $\mathscr{C}$ consists of the following data
\begin{itemize}
\item[(1)]
A set $\left|\mathscr{C}\right|$ of objects.
\item[(2)]
For each pair $x,y\in\left|\mathscr{C}\right|$, a set of morphisms $\mathscr{C}(x,y)$ from $x$ to $y$.
\item[(3)]
For each triple $x,y,z\in\left|\mathscr{C}\right|$, a function
$$
\mathscr{C}(x,y)\times \mathscr{C}(y,z)\longrightarrow\mathscr{C}(x,z), \ \ (f,g)\longmapsto fg.
$$
which is called composition and satisfies the associativity axiom.
\end{itemize}

Let $\mathscr{C}$ and $\mathscr{D}$ be two semicategories. A semifunctor $F:\mathscr{C}\rightarrow\mathscr{D}$ from $\mathscr{C}$ to $\mathscr{D}$ consists of the following data
\begin{itemize}
\item[(1)]
A function $\left|\mathscr{C}\right|\rightarrow\left|\mathscr{D}\right|$.
\item[(2)]
For any $x,y\in\left|\mathscr{C}\right|$, a function $\mathscr{C}(x,y)\rightarrow\mathscr{D}(Fx,Fy)$.
\end{itemize}
These data are compatible with compositions in the following manner: $F(fg)=F(f)F(g)$ when $f,g$ are composable.
\end{definition}

\begin{definition1}\label{A33}
Hence, we can form a list of semicategories as follows:
\begin{enumerate}
\item
The semicategory $\textrm{\underline{Perf}}_{S}$ whose objects are perfect algebraic spaces over $S$ and morphisms are perfect morphisms in $AP_{S}$.
\item
The semicategory $\textrm{\underline{QPerf}}_{S}$ whose objects are quasi-perfect algebraic spaces over $S$ and morphisms are perfect morphisms in $AP_{S}$.
\item
The semicategory $\textrm{\underline{SPerf}}_{S}$ whose objects are semiperfect algebraic spaces over $S$ and morphisms are perfect morphisms in $AP_{S}$.
\item
The semicategory $\textrm{\underline{StPerf}}_{S}$ whose objects are strongly perfect algebraic spaces over $S$ and morphisms are perfect morphisms in $AP_{S}$.
\item
The semicategory $\widehat{\textrm{\underline{Perf}}_{S}}$ whose objects are representable perfect algebraic spaces over $S$ and morphisms are perfect morphisms in $AP_{S}$.
\item
The semicategory $\widehat{\textrm{\underline{QPerf}}_{S}}$ whose objects are representable quasi-perfect algebraic spaces over $S$ and morphisms are perfect morphisms in $AP_{S}$.
\item
The semicategory $\widehat{\textrm{\underline{SPerf}}_{S}}$ whose objects are representable semiperfect algebraic spaces over $S$ and morphisms are perfect morphisms in $AP_{S}$.
\item
The semicategory $\widehat{\textrm{\underline{StPerf}}_{S}}$ whose objects are representable strongly perfect algebraic spaces over $S$ and morphisms are perfect morphisms in $AP_{S}$.
\end{enumerate}
\end{definition1}

Then there are strings of inclusion semifunctors
\begin{align}
&\widehat{\textrm{\underline{Perf}}_{S}}\longrightarrow\textrm{\underline{Perf}}_{S}\longrightarrow\textrm{Perf}_{AP_{S}} \label{I6} \\
&\widehat{\textrm{\underline{QPerf}}_{S}}\longrightarrow\textrm{\underline{QPerf}}_{S}\longrightarrow\textrm{QPerf}_{AP_{S}} \\
&\widehat{\textrm{\underline{SPerf}}_{S}}\longrightarrow\textrm{\underline{SPerf}}_{S}\longrightarrow\textrm{SPerf}_{AP_{S}} \\
&\widehat{\textrm{\underline{StPerf}}_{S}}\longrightarrow\textrm{\underline{StPerf}}_{S}\longrightarrow\textrm{StPerf}_{AP_{S}} \\
&\textrm{\underline{StPerf}}_{S}\longrightarrow\textrm{\underline{QPerf}}_{S}\longrightarrow\textrm{\underline{SPerf}}_{S} \\
&\widehat{\textrm{\underline{Perf}}_{S}}\longrightarrow\widehat{\textrm{\underline{StPerf}}_{S}}\longrightarrow\widehat{\textrm{\underline{QPerf}}_{S}}\longrightarrow\widehat{\textrm{\underline{SPerf}}_{S}}
\end{align}

At the end of this section, we would like to use three big commutative diagrams to illustrate the relationships of all categories and semicategories defined above. Figure \ref{F1} and Figure \ref{F2} illustrate the inclusion functors of categories in Definition \ref{A31} and \ref{A34}, while Figure \ref{F3} illustrates the inclusion semifunctors of semicategories in Definition \ref{A33}.

\begin{center}
$$
\begin{tikzcd}[row sep=1.5em, column sep = 1.5em]
\textrm{StPerf}_{AP_{S}} \arrow[rr]  &&
\textrm{QPerf}_{AP_{S}}  \arrow[rr] &&
\textrm{SPerf}_{AP_{S}}   \\
& \widehat{\textrm{StPerf}_{AP_{S}}} \arrow[ul] \arrow[rr] && \widehat{\textrm{QPerf}_{AP_{S}}} \arrow[ul] \arrow[rr] && \widehat{\textrm{SPerf}_{AP_{S}}} \arrow[ul]\\
\mathcal{ST}\textrm{Perf}_{S}^{**} \arrow[rr,]\arrow[uu]   && \mathcal{Q}\textrm{Perf}_{S}^{*} \arrow[rr] \arrow[uu]  && \mathcal{S}\textrm{Perf}_{S} \arrow[uu] \\
& \widehat{\mathcal{ST}\textrm{Perf}_{S}^{**}} \arrow[rr] \arrow[uu]   \arrow[ul]&&  \widehat{\mathcal{Q}\textrm{Perf}_{S}^{*}} \arrow[uu]  \arrow[ul] \arrow[rr] && \widehat{\mathcal{S}\textrm{Perf}_{S}}  \arrow[ul] \arrow[uu] \\
\mathcal{ST}\textrm{Perf}_{S}^{*} \arrow[rr]  \arrow[uu] && \mathcal{Q}\textrm{Perf}_{S} \arrow[rr] \arrow[uu]  && \mathcal{S}\textrm{Perf}_{S}^{\sharp}  \arrow[uu]  \\
& \widehat{\mathcal{ST}\textrm{Perf}_{S}^{*}}  \arrow[rr] \arrow[ul] \arrow[uu]&& \widehat{\mathcal{Q}\textrm{Perf}_{S}} \arrow[rr] \arrow[uu] \arrow[ul] && \widehat{\mathcal{S}\textrm{Perf}_{S}^{\sharp}} \arrow[ul] \arrow[uu] \\
\mathcal{ST}\textrm{Perf}_{S} \arrow[rr,]  \arrow[uu] && \mathcal{Q}\textrm{Perf}_{S}^{\sharp} \arrow[rr] \arrow[uu]  && \mathcal{S}\textrm{Perf}_{S}^{\sharp\sharp}  \arrow[uu]  \\
& \widehat{\mathcal{ST}\textrm{Perf}_{S}}  \arrow[rr] \arrow[ul] \arrow[uu]&& \widehat{\mathcal{Q}\textrm{Perf}_{S}\sharp} \arrow[rr] \arrow[uu] \arrow[ul] && \widehat{\mathcal{S}\textrm{Perf}_{S}^{\sharp\sharp}} \arrow[ul] \arrow[uu] \\
\end{tikzcd}
$$
\captionof{figure}{The commutative diagram of inclusion functors}
\label{F1}
\end{center}

\begin{center}
$$
\begin{tikzcd}[row sep=1.5em, column sep = 1.5em]
\widehat{\textrm{Perf}_{AP_{S}}}  \arrow[r] & \widehat{\textrm{StPerf}_{AP_{S}}}  \arrow[r] & \widehat{\textrm{QPerf}_{AP_{S}}} \arrow[r] & \widehat{\textrm{SPerf}_{AP_{S}}} \\
\widehat{\rm{Perf}^{\mathcal{S}}_{S}} \arrow[u] \arrow[r] & \widehat{\mathcal{ST}\textrm{Perf}_{S}^{**}}\arrow[u]  \arrow[r] & \widehat{\mathcal{Q}\textrm{Perf}_{S}^{*}} \arrow[u] \arrow[r] & \widehat{\mathcal{S}\textrm{Perf}_{S}} \arrow[u] \\
\widehat{\rm{Perf}^{\mathcal{Q}}_{S}} \arrow[u] \arrow[r] & \widehat{\mathcal{ST}\textrm{Perf}_{S}^{*}}\arrow[u]  \arrow[r] & \widehat{\mathcal{Q}\textrm{Perf}_{S}} \arrow[u] \arrow[r] & \widehat{\mathcal{S}\textrm{Perf}_{S}^{\sharp}} \arrow[u] \\
\widehat{\rm{Perf}^{\mathcal{Q}}_{ST}} \arrow[u] \arrow[r] & \widehat{\mathcal{ST}\textrm{Perf}_{S}}\arrow[u]  \arrow[r] & \widehat{\mathcal{Q}\textrm{Perf}_{S}^{\sharp}} \arrow[u] \arrow[r] & \widehat{\mathcal{S}\textrm{Perf}_{S}^{\sharp\sharp}} \arrow[u]
\end{tikzcd}
$$
\captionof{figure}{The commutative diagram of inclusion functors}
\label{F2}
\end{center}

\begin{center}
$$
\begin{tikzcd}[row sep=1.5em, column sep = 1.5em]
\textrm{StPerf}_{AP_{S}}  \arrow[r] & \textrm{QPerf}_{AP_{S}} \arrow[r] & \textrm{SPerf}_{AP_{S}} \\
\textrm{\underline{StPerf}}_{S}  \arrow[r]\arrow[u] & \textrm{\underline{QPerf}}_{S} \arrow[r]\arrow[u] & \textrm{\underline{SPerf}}_{S}\arrow[u] \\
\widehat{\textrm{\underline{StPerf}}_{S}}  \arrow[r]\arrow[u] & \widehat{\textrm{\underline{QPerf}}_{S}} \arrow[r]\arrow[u] & \widehat{\textrm{\underline{SPerf}}_{S}}\arrow[u]
\end{tikzcd}
$$
\captionof{figure}{The commutative diagram of inclusion semifunctors}
\label{F3}
\end{center}

\section{Perfect groupoids in algebraic spaces}\label{B6}
In this section, we briefly study the notions of perfect group algebraic spaces and perfect groupoids in algebraic spaces. First, we slightly loosen the notion of a relation $R\subset A\times A$ on a set $A$ and require it as a map $R\rightarrow A\times A$ in the setting of schemes or algebraic spaces (see \cite[Tag022P]{StackProject} or \cite{Mori}). Let $B$ be a base algebraic space over $S$. Here we specialize the definitions to the case that is of particular interest to us.
\begin{definition}
Let $U$ be an algebraic space over $B$. A pre-relation $j:R\rightarrow U\times_{B}U$ is said to be perfect if $U,R$ are perfect. Then a relation (resp. pre-equivalence relation, equivalence relation) $j:R\rightarrow U\times_{B}U$ is said to be perfect if it is perfect as a pre-relation.
\end{definition}

Next, we make the following definitions of perfect group algebraic spaces and perfect groupoids in algebraic spaces.
\begin{definition}
Let $(G,m)$ be a group algebraic space over $B$ and $(U,R,s,t,c)$ be a groupoid in algebraic spaces over $B$. We say that $(G,m)$ is perfect (resp. semiperfect) if $G$ is perfect (resp. semiperfect). And $(U,R,s,t,c)$ is said to be perfect (resp. semiperfect) if $U,R$ are perfect (resp. semiperfect).
\end{definition}

Perfect group algebraic spaces are stable under certain fibre products.
\begin{proposition}
Let $(F,m_{0}),(G,m_{1}),(H,m_{2})$ be group algebraic spaces over $B$ with homomorphisms $(F,m_{0})\rightarrow(H,m_{2})$ and $(G,m_{1})\rightarrow(H,m_{2})$. Suppose that $(H,m_{2})$ is semiperfect such that there exist perfect schemes $U,V\in\ObSchS$ and surjective \'{e}tale maps $h_{U}\rightarrow F,h_{V}\rightarrow G$ making $h_{U}\times_{H}h_{V}$ represented by a perfect scheme. If $(F,m_{0}),(G,m_{1})$ are perfect, then $(F\times_{H}G,m)$ is a perfect group algebraic space over $B$, where
$$m:(F\times_{H}G)\times_{B}(F\times_{H}G)\rightarrow F\times_{H}G$$
is a morphism of algebraic spaces over $B$.
\end{proposition}
\begin{proof}
It is easy to check that $(F\times_{H}G,m)$ is a group algebraic space over $B$ since the category of groups has fibre products. Then the statement follows from Proposition \ref{A7}.
\end{proof}

Meanwhile, perfect groupoids in algebraic spaces are stable under certain fibre products.
\begin{proposition}
Suppose that we are given the following data:
\begin{enumerate}
  \item
  Groupoids in algebraic spaces $(U,R,s,t,c),(U',R',s',t',c'),(U'',R'',s'',t'',c'')$ over $B$.
  \item
  Morphisms $(U,R,s,t,c)\rightarrow(U'',R'',s'',t'',c'')$ and $(U',R',s',t',c')\rightarrow(U'',R'',s'',t'',c'')$ of groupoids in algebraic spaces over $B$.
  \item
  $(U'',R'',s'',t'',c'')$ is semiperfect such that there exist perfect schemes $W,W',W'',W'''\in\ObSchS$ and surjective \'{e}tale maps $h_{W}\rightarrow U,h_{W'}\rightarrow U',h_{W''}\rightarrow R,h_{W'''}\rightarrow R'$ making $h_{W}\times_{U''}h_{W'},h_{W''}\times_{R''}h_{W'''}$ represented by perfect schemes.
\end{enumerate}
If $(U,R,s,t,c),(U',R',s',t',c')$ are perfect, then $(U\times_{U''}U',R\times_{R''}R',s''',t''',c''')$ is a perfect groupoid in algebraic spaces over $B$, where
$$s''',t''':R\times_{R''}R'\rightarrow U\times_{U''}U'\ {\rm and}\ c''':(R\times_{R''}R')\times_{s''',(U\times_{U''}U'),t'''}(R\times_{R''}R')\rightarrow R\times_{R''}R'$$
are morphisms of algebraic spaces over $B$.
\end{proposition}
\begin{proof}
It follows from \cite[\S3.2, Theorem 5]{Acosta} that $(U\times_{U''}U',R\times_{R''}R',s''',t''',c''')$ is a groupoid in algebraic spaces over $B$. Then the statement follows from Proposition \ref{A7}.
\end{proof}

It is clear that every perfect groupoid in algebraic spaces gives rise to a perfect pre-equivalence relation.
\begin{lemma}
Let $(U,R,s,t,c)$ be a perfect groupoid in algebraic spaces over $B$. Then the morphism $j:R\rightarrow U\times_{B}U$ is a perfect pre-equivalence relation.
\end{lemma}

Moreover, every perfect equivalence relation gives rise to a perfect groupoid in algebraic spaces.
\begin{lemma}
Let $j:R\rightarrow U\times_{B}U$ be a perfect equivalence relation over $B$. Then there is a unique way to extend it to a perfect groupoid in algebraic spaces $(U,R,s,t,c)$ over $B$.
\end{lemma}

\end{document}